\documentclass[a4paper]{article}

\usepackage{amsfonts}
\usepackage{graphicx,epsfig,subfigure}
\usepackage{amsmath,amsthm}
\usepackage{subfigure}
\usepackage{a4wide}
\usepackage{color}

 \def\2{C^{1,2}(\R\times\R^N)}

\def\e{\epsilon}
\def\eps{\epsilon}

\def\R{\mathbb{R}}

\def\epsilon{\varepsilon}

\def\ind{\mathbf{1}}

\newtheorem{thm}{\bf Theorem}
\newtheorem{lem}[thm]{\bf Lemma}
\newtheorem{prop}[thm]{\bf Proposition}

\newtheorem{defi}[thm]{\bf Definition}
\newtheorem{rmq}[thm]{\bf Remark}

\numberwithin{equation}{section}

\newcommand{\U}{\nu}
\newcommand{\w}{u}
\newcommand{\uu}{w}

\newcommand{\supp}{\mathrm{Supp}\,}
\newcommand{\minmod}{\mathrm{minmod}\,}
\newcommand{\sign}{\mathrm{sign}\,}

\begin{document}
\title{Hyperbolic traveling waves driven by growth}

\author{Emeric Bouin\footnote{Ecole Normale Sup\'erieure de Lyon, UMR CNRS 5669 'UMPA', 46 all\'ee d'Italie, F-69364~Lyon~cedex~07, France. E-mail: \texttt{emeric.bouin@ens-lyon.fr}}, Vincent Calvez\footnote{Corresponding author}\;\footnote{Ecole Normale Sup\'erieure de Lyon, UMR CNRS 5669 'UMPA', and INRIA project NUMED, 46 all\'ee d'Italie, F-69364~Lyon~cedex~07, France. E-mail: \texttt{vincent.calvez@ens-lyon.fr}}, Gr\'egoire Nadin\footnote{Universit\'e Pierre et Marie Curie-Paris 6, UMR CNRS 7598 'LJLL', BC187, 4 place de Jussieu, F-75252~Paris~cedex~05, France. E-mail: \texttt{nadin@ann.jussieu.fr}}}
\date{\today}
\maketitle

\begin{abstract}
We perform the analysis of a hyperbolic model which is the analog of the Fisher-KPP equation. This model accounts for particles that move at maximal speed $\eps^{-1}$ ($\eps>0$), and proliferate according to a reaction term of monostable type. We study the existence and stability of traveling fronts. We exhibit a transition depending on the parameter $\epsilon$: for small $\epsilon$ the behaviour is essentially the same as for the diffusive Fisher-KPP equation. However, for large $\epsilon$ the traveling front with minimal speed is discontinuous and travels at the maximal speed $\eps^{-1}$. The traveling fronts with minimal speed are linearly stable in weighted $L^2$ spaces. We also prove local nonlinear stability of the traveling front with minimal speed when $\eps$ is smaller than the transition parameter.
\end{abstract}



\section{Introduction}

We consider the problem of traveling fronts driven by growth ({\em e.g.} cell division) together with cell dispersal, where the motion process is given by a hyperbolic equation. This is motivated by the occurence of traveling pulses in populations of bacteria swimming inside a narrow channel \cite{Adler66,Saragosti1}. It has been demonstrated that kinetic models are well adapted to this problem \cite{Saragosti2}. We will focus on the following model introduced by Dunbar and Othmer \cite{DO} (see also Hadeler \cite{Hadeler}) and Fedotov \cite{Fedotov98,Fedotov99,Fedotov08}
\begin{equation} \label{eqprinc}
 \e^2 \partial_{tt}\rho_\e(t,x) +\left(1-\e^2 F'(\rho_\e(t,x))\right)\partial_t \rho_\e(t,x) - \partial_{xx} \rho_\e(t,x) =  F(\rho_\e(t,x))\, , \quad t>0\, , \quad x\in \R \, .
\end{equation}
The cell density is denoted by $\rho_\e(t,x)$. The parameter $\e>0$ is a scaling factor. It accounts for the ratio between the mean free path of cells and the space scale. The growth function $F$ is subject to the following assumptions (the so-called monostable nonlinearity)
\begin{equation}\label{eq:hyp F}
\left\{\begin{array}{l} 
F \in\mathcal{C}^{3} ([0,1])\,,\quad   F \hbox{ is uniformly strictly concave}:\;  \inf_{[0,1]} (- F'') =: \alpha >0 \,, \medskip\\ 
F(0)=F(1)=0\,,\quad F(\rho)>0\; \mbox{ if }\; \rho\in (0,1)\,.
\end{array}\right.
\end{equation}
For the sake of clarity we will sometimes take as an example the logistic growth function $F(\rho) = \rho(1- \rho)$.

Equation \eqref{eqprinc} is equivalent to the hyperbolic system 
\begin{equation}
\label{eq:telegraph intro}
\left\{
\begin{array}{l}
\partial_t \rho_\e + \e^{-1}\partial_x \left( j_\e \right) = F(\rho_\e) \\
\e \partial_t j_\e + \partial_x \rho_\e = -\e^{-1} j_\e \, . 
\end{array}
\right.
\end{equation}
The expression of $j_\e$ can be computed explicitly in terms of $\rho_\e$ as follows,
\begin{equation}
j_\e(t,x) = - \frac{1}{\e}\int_{0}^{t} \partial_x \rho_{\e}(s,x) \exp{\left( \frac{ s - t }{\e^2} \right) }\, ds+j_\eps(0,x)\,,
\end{equation}
but this expression will not be directly used afterwards. We will successively use the formulation \eqref{eqprinc} or the equivalent formulation \eqref{eq:telegraph intro}.

\bigskip

Since the pioneering work by Fisher \cite{Fisher37} and Kolmogorov-Petrovskii-Piskunov \cite{KPP}, dispersion of biological species has been usually modelled by mean of reaction-diffusion equations. The main drawback of these models is that they allow infinite speed of propagation. This is clearly irrelevant for biological species. Several modifications have been proposed to circumvent this issue. It has been proposed to replace the linear diffusion by a nonlinear diffusion of porous-medium type \cite{Shigesada80,Murray03a,Painter02}. This is known to yield propagation of the support at finite speed \cite{Nagai83a,Nagai83b}. The density-dependent diffusion coefficient stems for a pressure effect among individuals which influences the speed of diffusion. Pressure is very low when the population is sparse, whereas it has a strong effect when the population is highly densified. Recently, this approach has been developped for the invasion of glioma cells in the brain \cite{Aubert09}. Alternatively, some authors have proposed to impose a limiting flux for which the nonlinearity involves the gradient of the concentration \cite{Mazon06,Soler,Mazon11}. 

The diffusion approximation is generally acceptable in ecological problems where space and time scales are large enough. However, kinetic equations have emerged recently to model self-organization in bacterial population at smaller scales \cite{Ajmb80,ODA,ErbanOthmer04,aerotaxis,PerthameBook,Saragosti1,Saragosti2}. These models are based on velocity-jump processes. It is now standard to perform a drift-diffusion limit to recover classical reaction-diffusion equations \cite{HillenOthmer,CMPS,ErbanOthmer04,Hwang}. However it is claimed in \cite{Saragosti2} that the diffusion approximation  is not suitable, and the full kinetic equation has to be handled with. 
Equation \eqref{eqprinc} can be reformulated as a kinetic equation with two velocities only $v = \pm \eps^{-1}$ (see \eqref{eq:twospeeds} below). This provides a clear biological interpretation of equation (1.1) as a simple model for bacteria colonies where bacteria reproduce themselves, and move following a run-and-tumble process. 

%Another choice consists in deriving macroscopic equations from an underlying description at the kinetic level. The hyperbolic nature of the kinetic equation guarantees finite speed of propagation. Alt and co-workers have derived systematically kinetic models for cell dispersal based on velocity-jump processes (see also Hadeler, Gallay, Hillen for further analysis). Fedotov has proposed to couple this velocity-jump process with a logistic term accounting for growth of the population. 

%Kinetic models are suitable for bacterial collective motion (Saragosti et al., Erban, Aerotaxis). They have been proposed as a basis for the modelling of glioma invasion (Fedotov 1999, 2008). 

Hyperbolic models coupled with growth have already been studied in \cite{DO,Hadeler,GallayRaugel,Cuesta}. In \cite{Hadeler} it is required that the nonlinear function in front of the time first derivative $\partial_t \rho_\e$ is positive (namely here, $1-\e^2 F'(\rho)>0$). Indeed, this enables to perform a suitable change of variables in order to reduce to the classical Fisher-KPP problem. In our context this is equivalent to $\e^2 F'(0) <1$ since $F$ is concave. In \cite{GallayRaugel} this nonlinear contribution is replaced by 1: the authors study the following equation (damped hyperbolic Fisher-KPP equation),
\begin{equation*}
 \e^2 \partial_{tt}\rho_\e(t,x) +\partial_t \rho_\e(t,x) - \partial_{xx} \rho_\e(t,x) =  F(\rho_\e(t,x))\, .
\end{equation*}

We also refer to \cite{Cuesta} where the authors analyse a kinetic model more general than \eqref{eqprinc}. They develop a perturbative approach, close to the diffusive regime $\e\ll 1$. 

\bigskip

It is worth recalling some basic results related to reaction-diffusion equations. First, as $\e \to 0$ the density $\rho_\e$ solution to \eqref{eqprinc} formally converges to a solution of the Fisher-KPP equation \cite{Cuesta}:
\begin{equation*}%%\label{eq:FKPP}
\partial_t\rho_0(t,x) - \partial_{xx} \rho_0(t,x) = F(\rho_0(t,x))\, .
\end{equation*}
The long time behaviour of such equation is well understood since the pioneering works by Kolmogorov-Petrovsky-Piskunov \cite{KPP} and Aronson-Weinberger \cite{AW}. For nonincreasing initial data with sufficient decay at infinity the solution behaves asymptotically as a traveling front moving at the speed $s = 2\sqrt{F'(0)}$. Moreover the traveling front solution with minimal speed is stable in some $L^2$ weighted space \cite{Gallay94}.

In this work we prove that analogous results hold true in the {\em parabolic regime} $\e^2 F'(0) <1$. Namely there exists a continuum of speeds $[s^*(\eps),\e^{-1})$ for which \eqref{eqprinc} admits smooth traveling fronts. The minimal speed is given by \cite{Fedotov98}
\begin{equation}\label{eq:minspeed 1}
s^*(\e)= \frac{2\sqrt{F'(0)}}{1+\e^2 F'(0)}\, , \quad \hbox{if } \e^2 F'(0) < 1\, .
\end{equation}
Obviously we have $s^*(\eps) \leq \min(2\sqrt{F'(0)},\e^{-1})$. There also exists {\em supersonic} traveling fronts, with speed $s>\eps^{-1}$. This appears surprising at first glance since the speed of propagation for the hyperbolic equation \eqref{eqprinc} is $\eps^{-1}$ (see formulation \eqref{eq:telegraph intro} and Section \ref{sec:num}). These fronts are essentially driven by growth, since they travel faster than the maximum speed of propagation. The results are summarized in the following Theorem.

\begin{thm}[Parabolic regime] \label{thm-diffusion}
Assume that $\e^2 F'(0)< 1$. The following alternatives hold:
\begin{description}
\item[(a)] There exists no smooth or weak traveling front of speed $s\in [0,s^*(\e))$.
\item[(b)] For all $s\in [s^*(\e),\e^{-1})$, there exists a smooth traveling front solution of \eqref{eqprinc} with speed $s$. 
\item[(c)] For $s = \e^{-1}$ there exists a weak traveling front.
\item[(d)] For all $s\in (\e^{-1},\infty)$ there also exists a smooth traveling front of speed $s$.
\end{description}
\end{thm}

We also obtain that the minimal speed traveling front is nonlinearly locally stable in the parabolic regime $\e^2 F'(0)< 1$ (see Section \ref{sec:NLstab}, Theorem \ref{thm:NLstab}).

There is a transition occuring when $\e^2 F'(0) = 1$. In the  {\em hyperbolic regime} $\e^2 F'(0) \geq 1$ the minimal speed speed becomes:
\begin{equation}\label{eq:minspeed 2}
s^*(\eps) = \e^{-1} \, , \quad \hbox{if } \e^2 F'(0) \geq 1\, .
\end{equation}
On the other hand, the front traveling with minimal speed $s^*(\eps)$ is discontinuous as soon as $\e^2 F'(0) > 1$. In the critical case $\e^2 F'(0) =1$ there exists a    continuous but not smooth traveling front with minimal speed $s^* = \sqrt{F'(0)}$. 

\begin{thm}[Hyperbolic regime] \label{thm-hyperbolic}
Assume that $\e^2 F'(0)\geq 1$. The following alternatives hold:
\begin{description}
\item[(a)] There exists no smooth or weak traveling front of speed $s\in [0,s^*(\e))$.
\item[(b)] There exists a weak traveling front solution of \eqref{eqprinc} with speed $s^*(\eps)=\e^{-1}$. The wave is discontinuous if $\e^2 F'(0)> 1$. 
\item[(c)] For all $s\in (\e^{-1},\infty)$ there exists a smooth traveling front of speed $s$. 
\end{description}
\end{thm}

%We also pay attention to the supersonic fronts having speed $s\geq \e^{-1}$.

%Equation \eqref{eqprinc} exhibits a transition in behaviour between the two following regime. If $\e^2 F'(0) <1$, what we call the {\em parabolic regime}, there exist smooth traveling front solutions, with speed above the minimal speed $s^*(\eps)$ \eqref{eq:minspeed}. We are able to prove that the traveling fronts having the minimal speed are locally nonlinearly stable as soon as $\eps < 1$. On the contrary, if $\e2 F'(0) >1$, what we call the {\em hyperbolic regime}, there exist  discontinuous traveling fronts solutions, with speed above the minimal speed $s^*(\eps) = \eps^{-1}$. 

\bigskip 

We conclude this introduction by giving the precise definition of traveling fronts (smooth and weak) that will be used throughout the paper. 

% It is important to define clearly the notion of traveling fronts, especially in the case of discontinuous profiles.

\begin{defi} \label{def-smoothtw}
We say that a function $\rho(t,x)$ is a smooth traveling front solution with speed $s$ of equation \eqref{eqprinc} if it can 
be written $\rho(t,x)=\U(x-st)$, where $\U\in\mathcal{C}^2(\R)$, $\U\geq 0$, $\U(-\infty)=1$, $\U(+\infty)=0$ and $\U$ satisfies
\begin{equation} \label{eqtw}
 (\e^2 s^2-1)\U''(z) - \left(1 - \e^2 F'(\U(z)) \right) s \U'(z) =  F(\U(z))\, , \quad z\in  \R\, .
\end{equation}
We say that $\rho$ is a weak traveling front with speed $s$ if it can 
be written $\rho(t,x)=\U(x-st)$,  where $\U\in L^\infty(\R)$, $\U\geq 0$, $\U(-\infty)=1$, $\U(+\infty)=0$ and $\U$ satisfies \eqref{eqtw} in the sense of distributions: 
\begin{equation*} %% \label{eqtwdistrib}
\forall \varphi \in \mathcal{D}(\R)\; , \quad  \int_\R \Big((\e^2 s^2-1)\U \varphi'' + \left(  \U - \e^2 F(\U) \right)s \varphi' - F(\U)\varphi \Big) dx = 0\, .
\end{equation*}
\end{defi}

%Such a function $r$ clearly satisfies (\ref{eqprinc}). 

%\begin{rmq}
%The bound $\e^{-1}$ corresponds to the maximal speed of propagation of the hyperbolic model \eqref{eq:telegraph}. We also have $s^{*}<2\sqrt{F'(0)}$ which is the critical speed of the diffusive limit (namely the Fisher-KPP equation).
%\end{rmq}

%When $\e^2 F'(0)\geq 1$, we enter in the hyperbolic regime and we need to consider non-smooth traveling fronts. 

%(see also Section \ref{sec:num}). We call {\em supersonic fronts} the one having speed $s>\e^{-1}$. They are essentially driven by growth. The above Theorems claim that there exist smooth supersonic fronts in any cases (Section \ref{sec:supersonic}).

In the following Section \ref{sec:num} we show some numerical simulations in order to illustrate our results. 
Section \ref{sec:TWexistence} is devoted to the proof of existence of the traveling fronts in the various regimes (resp. parabolic, hyperbolic, and supersonic). 
Finally, in Section \ref{sec:linstab} and Section \ref{sec:NLstab} we prove the stability of the traveling fronts having minimal speed $s^*(\eps)$. We begin with linear stability (Section \ref{sec:linstab}) since it is technically better tractable, and it let us discuss the case of the hyperbolic regime. We prove the full nonlinear stability in the range $\eps \in (0,1/\sqrt{F'(0)})$ (parabolic regime) in Section~\ref{sec:NLstab}.

\section{Numerical simulations}
\label{sec:num}

In this Section we perform numerical simulations of \eqref{eqprinc}. We choose a logistic reaction term: $F(\rho) = \rho(1-\rho)$. 
We first symmetrize the hyperbolic system \eqref{eq:telegraph intro} by introducing $f^+ = \frac12(\rho + j)$ and $f^- = \frac12(\rho - j)$. This results in the following system:
\begin{equation}
\label{eq:twospeeds}
%\left\{
\begin{cases}
\partial_t f^+(t,x) + \e^{-1}\partial_x f^+(t,x) = \dfrac{\e^{-2}}{2 } \left( f^-(t,x) - f^+(t,x) \right)+  \dfrac12 F(\rho(t,x)) \medskip\\
 \partial_t f^-(t,x) - \e^{-1}\partial_x f^-(t,x) = \dfrac{\e^{-2}}{2 } \left( f^+(t,x) - f^-(t,x) \right) +  \dfrac12 F(\rho(t,x))\, . 
\end{cases}
%\right.
\end{equation}
In other words, the population is split into two subpopulations: $\rho = f^+ + f^-$, where the density $f^+$ denotes particles moving to the right with velocity $\e^{-1}$, whereas $f^-$ denotes particles moving to the left with the opposite velocity. 
%We have $\rho = f^+ + f^-$, and the flux is $j = f^+ - f^-$.

We discretize the transport part using a finite volume scheme. Since we want to catch discontinuous fronts in the hyperbolic regime $\e^2F'(0)> 1$, we aim to avoid numerical diffusion. Therefore we use a nonlinear flux-limiter scheme \cite{Godlewski.Raviart,DespresLagoutiere}. The reaction part is discretized following the Euler explicit method.
\[  f^+_{n+1,i} = f^+_{n,i} - \e^{-1}\dfrac{\Delta t}{\Delta x}\left( f^+_{n,i} + p_i \dfrac{\Delta x}2 - f^+_{n,i-1} - p_{i-1}\dfrac{\Delta x}2 \right) + \e^{-2}\frac{\Delta t }{2 } \left( f_{n,i}^{-} - f_{n,i}^{+} \right) + \dfrac{\Delta t}{2} F(\rho_{n,i})\, .  \]
The non-linear reconstruction of the slope is given by 
\[ p_i = \minmod\left( \dfrac{f_{n,i}^{+} - f_{n,i-1}^{+}}{\Delta x}, \dfrac{f_{n,i+1}^{+} - f_{n,i}^{+}}{\Delta x} \right)\, , \quad \minmod(p,q) = \begin{cases} 0 & \mbox{if} \; \sign(p) \neq \sign(q) \\
\min(|p|,|q|)\sign(p) & \mbox{if} \; \sign(p) = \sign(q) \end{cases} \]
We compute the solution on the interval $(a,b)$ with the following boundary conditions: $f^+(a) = 1/2$ and $f^-(b) = 0$. 
The discretization of the second equation for $f^-$ \eqref{eq:twospeeds} is similar. The CFL condition reads $ \Delta t < \e \Delta x $. It degenerates when $\eps\searrow 0$, but we are mainly interested in the hyperbolic regime when $\eps$ is large enough. Other strategies should be used in the diffusive regime $\eps \ll 1$, e.g. asymptotic-preserving schemes (see \cite{FilbetJin,Carrillo.Goudon.Lafitte} and references therein).

Results of the numerical simulations in various regimes (parabolic and hyperbolic) are shown in Figure~\ref{fig:numerics}.

\begin{figure}
\begin{center}
\subfigure[$\e = 1/2$]
{
\includegraphics[width = .45\linewidth]{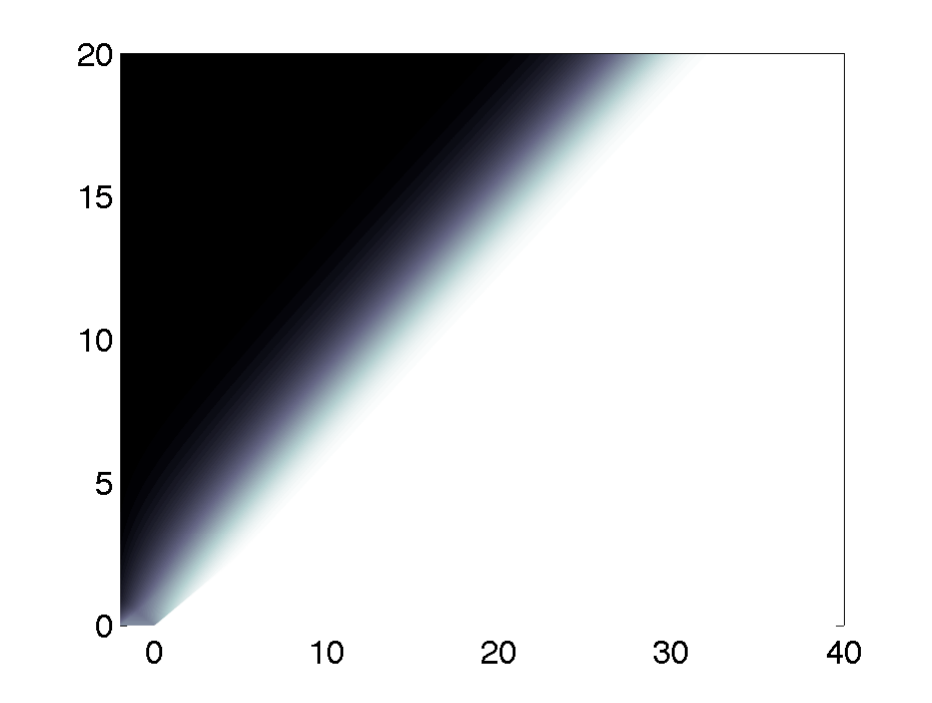} \; 
\includegraphics[width = .45\linewidth]{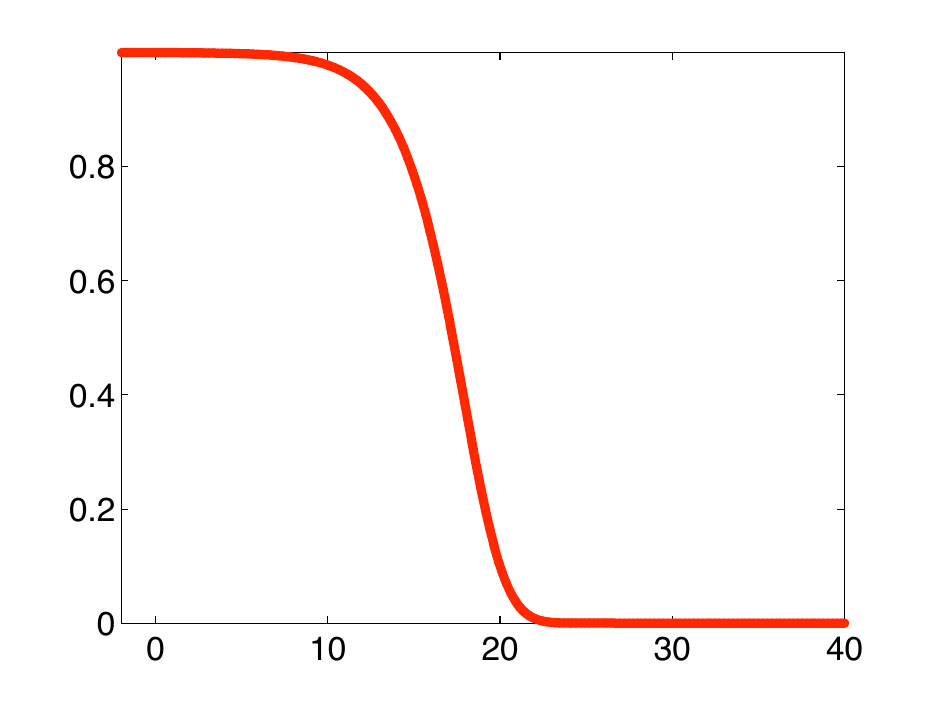}
}  
\subfigure[$\e = 1$]
{
\includegraphics[width = .45\linewidth]{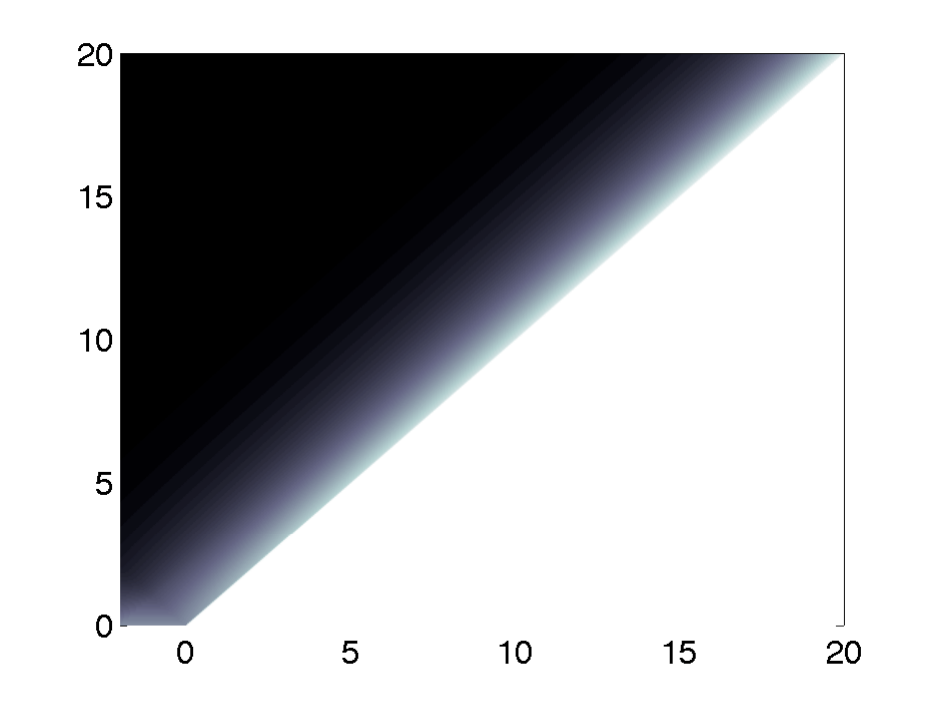} \; 
\includegraphics[width = .45\linewidth]{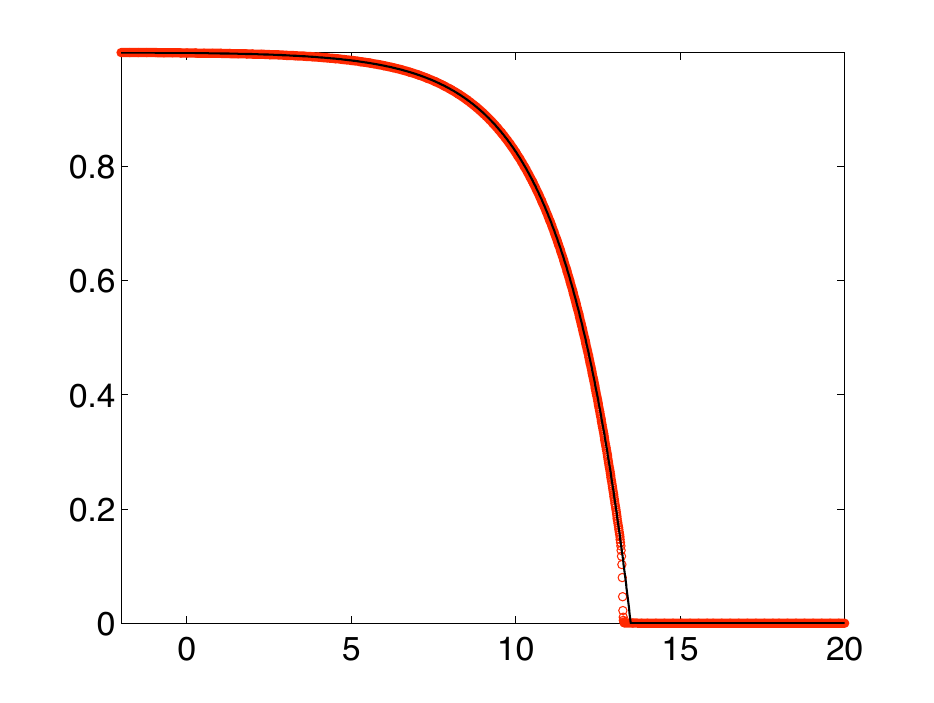}
}
\subfigure[$\e = 2$]{
\includegraphics[width = .45\linewidth]{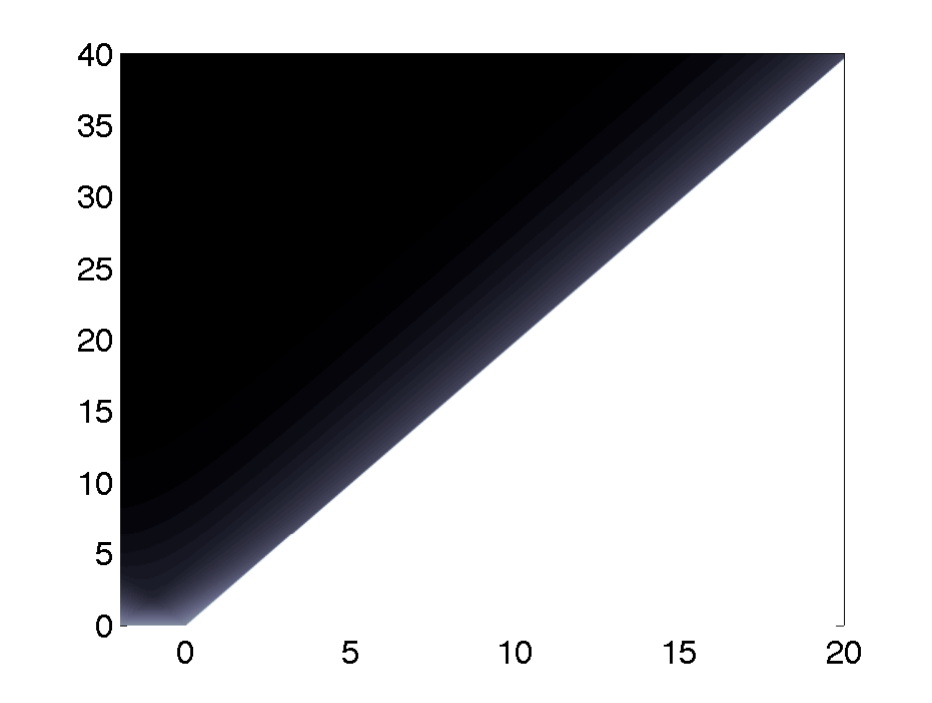} \; 
\includegraphics[width = .45\linewidth]{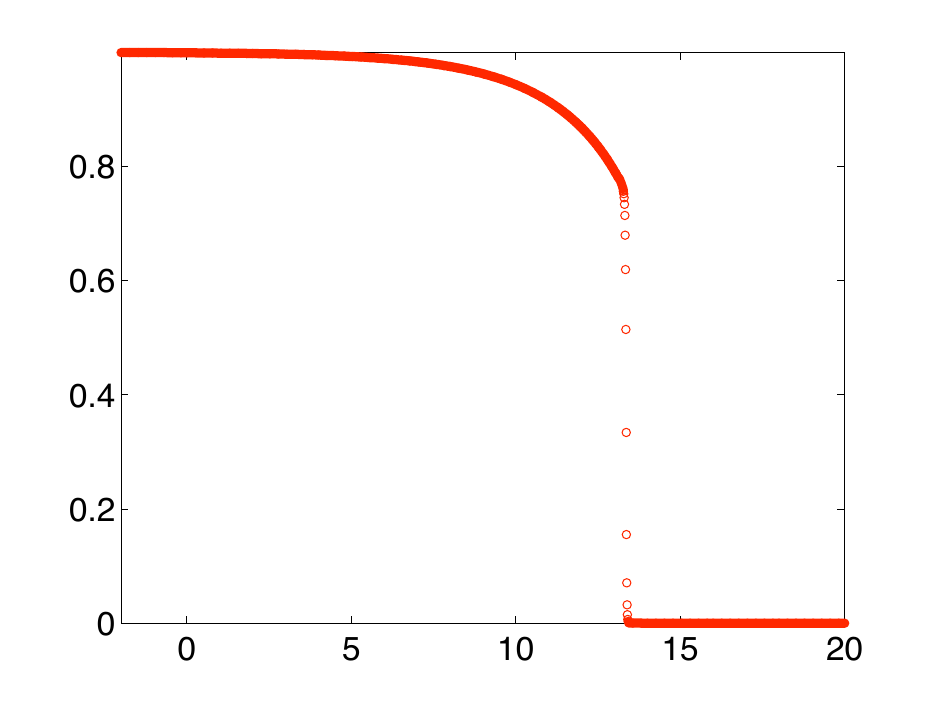}
}
\caption{Numerical simulations of the equation \eqref{eqprinc} for $F(\rho) = \rho(1-\rho)$ and for different values of $\eps = 0.5, 1, 2$. Numerical method is described in Section \ref{sec:num}. The initial data is the step function $f^+(x<0) = 1$, $f^+(x>0) = 0$, and $f^- \equiv 0$. For each value of $\eps$ we plot the density function $\rho = f^+ + f^-$ in the $(x,t)$ space, and the density $\rho(t_0,\cdot)$ at some chosen time $t_0$. We clearly observe in every cases a front traveling asymptotically at speed $s^*(\e)$ as expected. We also observe the transition between a smooth front and a discontinuous one. The transition occurs at $\eps = 1$. In the case $\eps = 1$ we have superposed the expected profile $\nu(z) = \left(1 -e^{z/2}\right)_+$ in black, continous line, for the sake of comparison.}\label{fig:numerics}
\end{center}
\end{figure}

%%%%%%%%%%%%%%%%%%%%%%%%%%%%%%%%%%%%%%%%%%%%%%%%%%%%%%%%%%%%%%%%%%%%%%%%%%%%%%%%%

\section{Traveling wave solutions: Proof of Theorems \ref{thm-diffusion} and \ref{thm-hyperbolic} }

\label{sec:TWexistence}

\subsection{Characteristic equation}

\label{ssec:phase}

We begin with a careful study of the linearization of \eqref{eqtw} around $\U\approx 0$. We expect an exponential decay $e^{-\lambda z}$ as $z\to +\infty$. The characteristic equation  reads as follows,
\begin{equation} (\e^2 s^2 - 1) \lambda^2 +  (1 - \e^2 F'(0)) s  \lambda - F'(0) = 0\, . 
\label{eq:charac}
\end{equation}
The discriminant is $\Delta = \left(\e^2 F'(0) + 1\right)^2 s^2 - 4 F'(0)$. Hence we expect an oscillatory behaviour in the case $\Delta <0$, i.e. $ s < s^*(\e) $. We assume henceforth $s\geq s^*(\e)$. In the case $s < \e^{-1}$ ({\em subsonic fronts}) we have to distinguish between the {\em parabolic regime} $\e^2 F'(0) < 1 $ and the {\em hyperbolic regime} $\e^2 F'(0) > 1$. In the former regime equation \eqref{eq:charac} possesses two positive roots, accounting for a damped behaviour. In the latter regime equation  \eqref{eq:charac} possesses two negative roots. In the case $s> \e^{-1}$ ({\em supersonic fronts}) we  get two roots having opposite signs.

Next we investigate the linear behaviour close to $\U\approx1$. We expect an exponential relaxation $1 - e^{\lambda' z}$ as $z\to -\infty$.
The characteristic equation reads as follows,
\begin{equation} (\e^2 s^2 - 1) \lambda'^2 -   (1 - \e^2 F'(1) )  s      \lambda' - F'(1) = 0\, . 
\label{eq:charac2}
\end{equation}
We have $\Delta' = [\e^2 F'(1) + 1]^2 s^2  - 4 F'(1)>0$. In the case $s < \e^{-1}$ equation \eqref{eq:charac2} possesses two roots having opposite signs. In the case $s > \e^{-1}$ it has two positive roots.

 We summarize our expectations about the possible existence of nonnegative traveling fronts in Table \ref{tab:phase}.
 \begin{table}
 \begin{center}
 \begin{tabular}{c|c|c}
   & $s < \e^{-1}$ & $s>\e^{-1}$ \\
 \hline  & & \\
 parabolic 
 & \begin{minipage}{0.25\linewidth} 
 \begin{center} if $s< s^*(\e)$, NO  \\ \includegraphics[width = \linewidth]{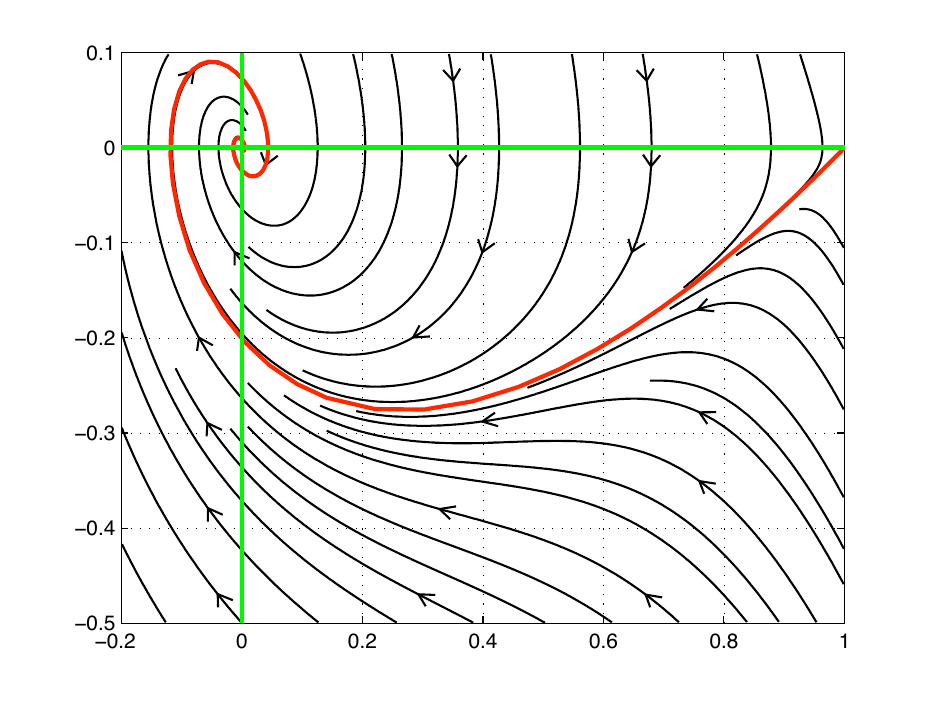} \end{center}\end{minipage}\; \begin{minipage}{0.25\linewidth}\begin{center} if $s\geq s^*(\e)$, YES \\ \includegraphics[width = \linewidth]{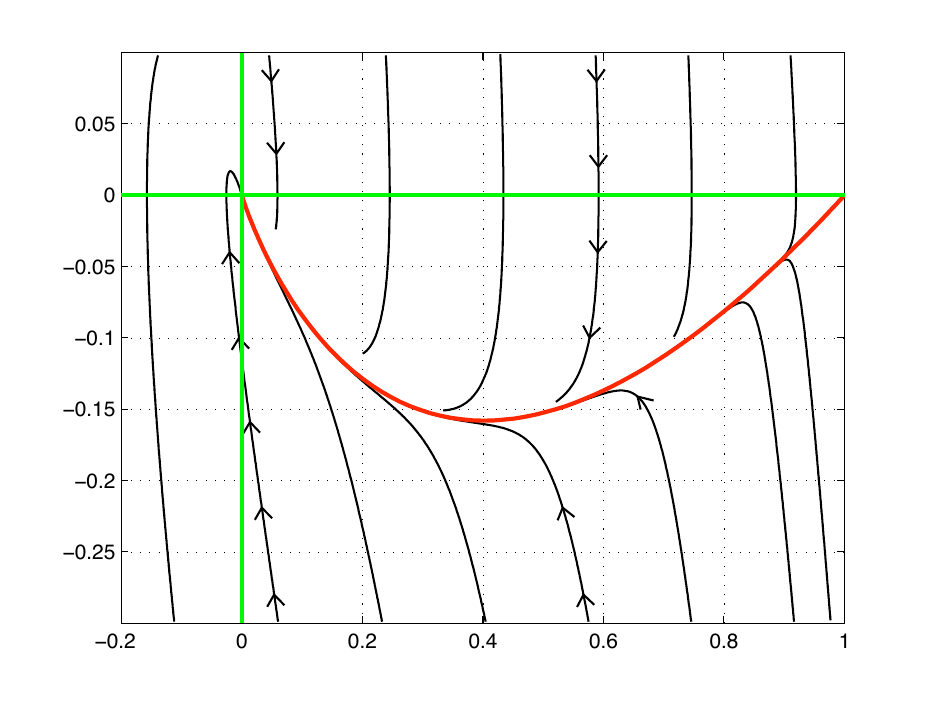}\end{center}\end{minipage} 
 & \begin{minipage}{0.25\linewidth}\begin{center}  YES \\ \includegraphics[width = \linewidth]{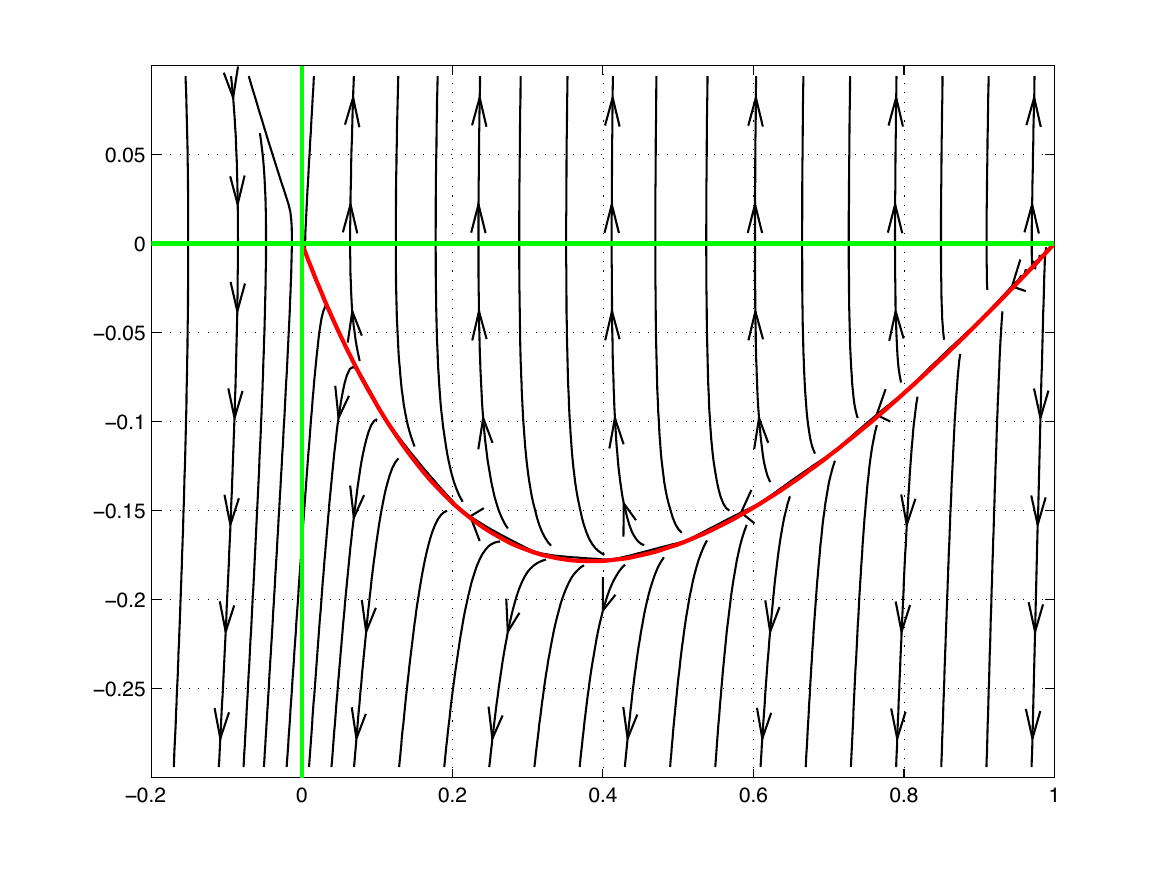}\end{center}\end{minipage}\\
 \hline & & \\
 hyperbolic 
 & \begin{minipage}{0.25\linewidth}\begin{center}  NO \\ \includegraphics[width = \linewidth]{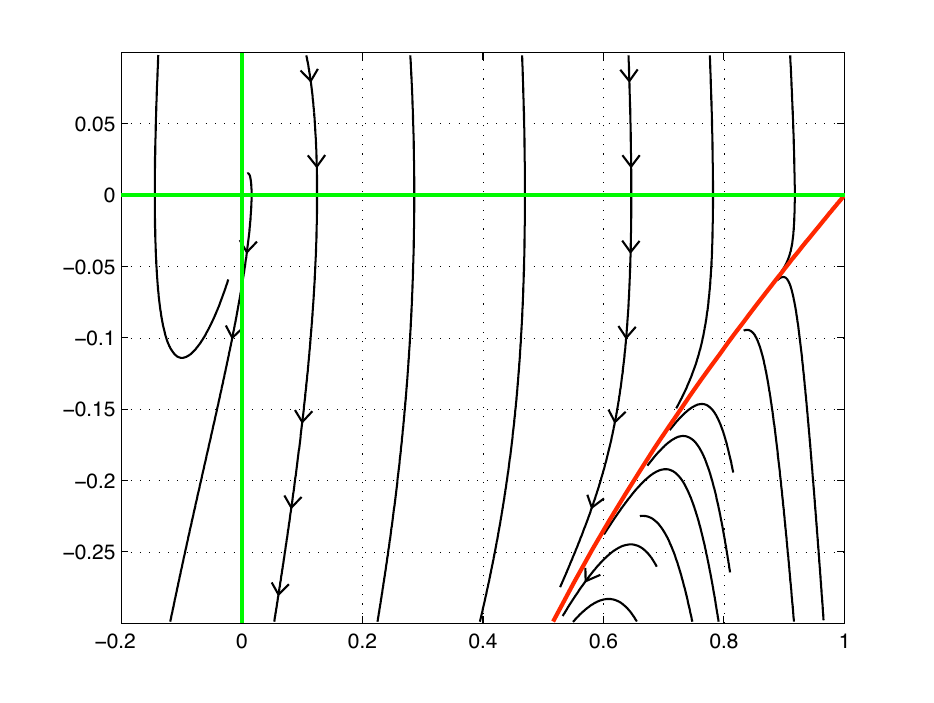}\end{center}\end{minipage}
 & \begin{minipage}{0.25\linewidth}\begin{center}  YES \\ \includegraphics[width = \linewidth]{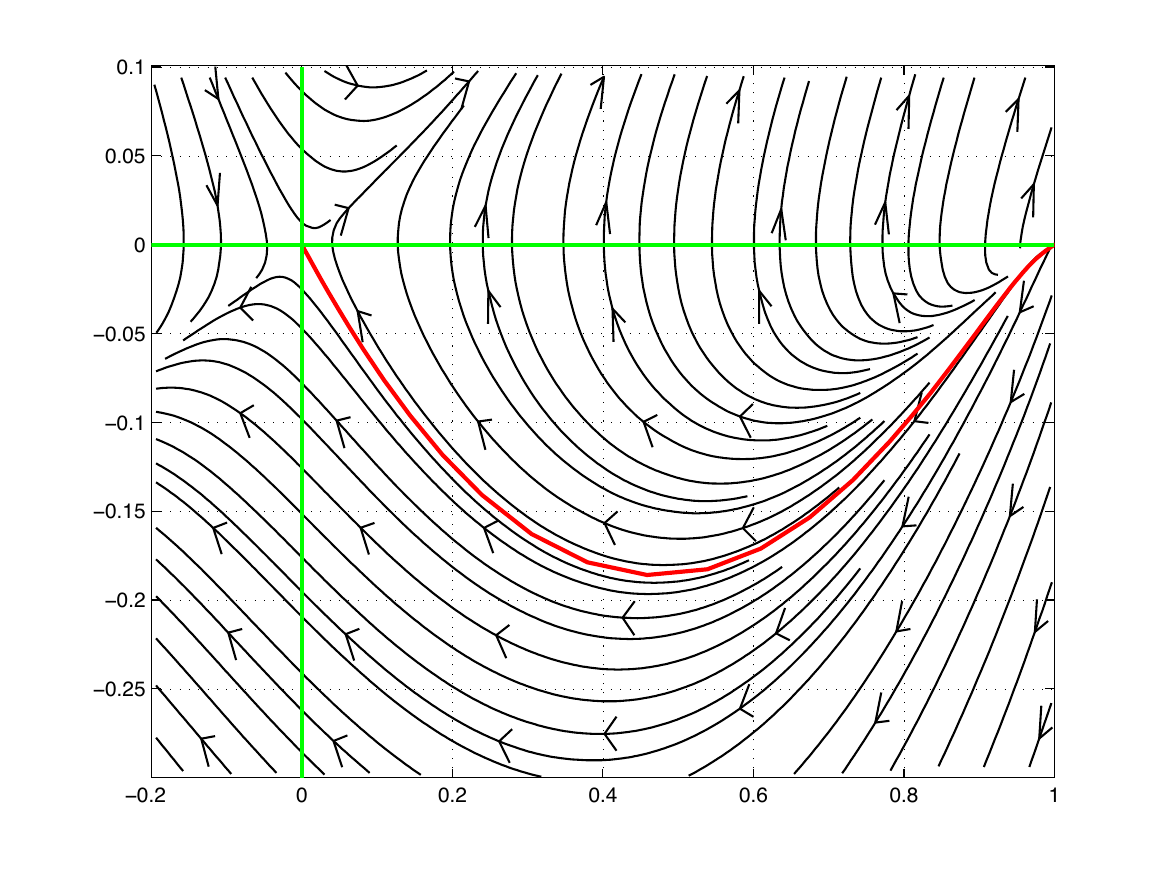}\end{center}\end{minipage}
 \end{tabular}
 \end{center}
 \caption{Phase plane dynamics depending on the regime (parabolic vs. hyperbolic) and the value of the speed with respect to $s^*(\e)$ and $\e^{-1}$. In every picture the red line represents the traveling front trajectory, and the green lines are the axes $\{u = 0\}$ and $\{v = 0\}$. We do not consider the case $s = \e^{-1}$ since the dynamics are singular in this case and should be considered separately (see Section \ref{ssec:weakTW}).}
 \label{tab:phase}
 \end{table}

\subsection{Proof of Theorems \ref{thm-diffusion}.(a) and \ref{thm-hyperbolic}.(a): Obstruction for $s < s^*(\e)$}

In this section we prove that no traveling front solution exists if the speed is below  $s^*(\e)$.

\begin{prop} \label{prop-nonex}
There exists no traveling front with speed $s$ for
$s< s^*(\e)$, where $s^*(\e)$ is given by \eqref{eq:minspeed 1}-\eqref{eq:minspeed 2}.
\end{prop}

\begin{rmq}
 Note that the proof below works in both cases $\e^2 F'(0)<1$ and $\e^2 F'(0)\geq 1$. 
\end{rmq} 

\begin{proof} We argue by contradiction. The obstruction comes from the exponential decay at $+\infty$. 
%Assume that such a wave $\rho$ exists. If $s<\frac{2\sqrt{F'(0)}}{1+\e^2 F'(0)})$ then $s<\frac{1}{\e}$ is also satisfied. 
Assume that there exists such a traveling front $\nu(z)$. 
As $s<s^*(\e)$, one has $s<\e^{-1}$ in the parabolic as well as in the hyperbolic regime. Hence, as $\nu$ is bounded and 
satisfies the elliptic equation 
\eqref{eqtw} in the sense of distributions, classical regularity estimates show that $\U$ is smooth. It is necessarily decreasing as soon at it is below $1$. 
Otherwise, it would reach a local minimum at some point $z_0\in\R$, for which $\U(z_0)<1$, $\U'(z_0) = 0$ and $\U''(z_0)\geq 0$. It would then follow from
\eqref{eqtw} that $F(\U(z_0))\leq 0$ and thus $\U(z_0)=0$. As $F\in\mathcal{C}^1 ([0,1])$, the Cauchy-Lipschitz theorem would imply $\U\equiv 0$, a contradiction.

Next, we define the exponential rate of decay at $+\infty$: 
\begin{equation*}
\lambda:= \liminf_{z\to +\infty} \frac{-\U'(z)}{\U (z)}\geq 0\, .
\end{equation*}
Consider a sequence $z_n\to+\infty$ such that $-\U'(z_n)/\U(z_n)\to \lambda$ and define the renormalized shift:
\[
\U_n(z):=\frac{\U (z+z_n)}{\U(z_n)}\,.
\]
This function is locally bounded by classical Harnack estimates. It satisfies
\[(\e^2 s^2-1)\U_n''(z) +\left(\e^2 F'(\U(z+z_n))-1\right)s \U_n'(z) =  \frac{1}{\U (z_n)}F(\U(z_n)\U_n(z))\, , \quad z\in \R\,.\]
As $F\in\mathcal{C}^1 ([0,1])$, $F(0)=0$ and $F$ is concave, the functions 
$z\mapsto \left(\e^2 F'(\U(z+z_n))-1\right)s$ and $z\mapsto \frac{1}{\U (z_n)}F(\U(z_n)\U_n(z))$ are uniformly bounded, uniformly in $n$. Hence, 
Schauder elliptic regularity estimates yield that the sequence $(\U_n)_n$ is locally bounded in the H\"older space $\mathcal{C}^\alpha (K)$ for any 
compact subset $K\subset \R$ and any $\alpha \in (0,1)$. The Ascoli theorem and a diagonal extraction process give an extraction, that we still denote
$(\U_n)_n$, such that $(\U_n)_n$ converges to some function $\U_\infty$ in $\mathcal{C}^\alpha (K)$ for any 
compact subset $K\subset \R$ and any $\alpha \in (0,1)$. The limiting function is a solution in the sense of distributions of
\begin{equation} \label{eq-nonexuinfty} (\e^2 s^2-1)\U_\infty''(z) +\left(\e^2 F'(0)-1\right)s \U_\infty'(z) =  F'(0)\U_\infty(z)\, , \quad z\in \R\, .\end{equation}
As this equation is linear, one has $\U_\infty \in\mathcal{C}^\infty (\R)$. If $\U_\infty (z_0)=0$, then as $\U_\infty$ is nonnegative, one would get $\U_\infty'(z_0)=0$ and thus
$\U_\infty\equiv 0$ by uniqueness of the Cauchy problem, which would be a contradiction since $\U_\infty (0)=\lim_{n\to +\infty} \U_n (0)=1$. Thus 
$\U_\infty$ is positive. 

Define $V=\U_\infty'/\U_\infty$. The definition of $\lambda$ yields 
$\min_{\R} V= V(0) =-\lambda$. Thus $V'(0) = 0$. 
%The function $V$ satisfies the following ODE
%\[(\e^2 s^2-1)V''+(\e^2 s^2-1)V V' +\left(\e^2 F'(0)-1\right)s V' =  0\, , \quad z\in \R\,.\]
%Hence, the strong maximum principle gives, $z\equiv -\lambda$. In other words, $u_\infty(x)=u_\infty(0) e^{-\lambda x}$ for all $x\in\R$.
Hence we deduce from \eqref{eq-nonexuinfty} that $\U_\infty(z) = \U_\infty(0) e^{-\lambda z}$. Plugging this into \eqref{eq-nonexuinfty}, 
we obtain that $\lambda$ satisfies the following second order equation,
\[(\e^2 s^2-1)\lambda^2 -(\e^2 F'(0)-1)s \lambda -  F'(0)=0\, .\] We know from Section \ref{ssec:phase} that, both in the parabolic and hyperbolic regimes, 
there is no real root in the case $s< s^*(\e)$. 
%There is obviously a positive root in the case $s>\e^{-1}$. In the case $s< \e^{-1}$, the existence of a positive root is equivalent to $s(1+\e^2 F'(0))\geq 2\sqrt{F'(0)}$ and $\e^2 F'(0)<1$.
%The existence of a nonnegative root for this polynomial equation is equivalent to 
%$s\geq \frac{2\sqrt{F'(0)}}{1+\e^2 F'(0)}$ and $\e^2 F'(0)<1$. 
\end{proof}

%%%%%%%%%%%%%%%%%%%%%%%%%%%%%%%%%%%%%%%%%%%%%%%%%%%%

\subsection{Proof of Theorem \ref{thm-diffusion}.(b): Existence of smooth traveling fronts in the parabolic regime $s\in [s^*(\e),\e^{-1})$}

\label{ssec:sub-sup}

In \cite{Hadeler} the author proves the existence of traveling front, by reducing the problem to the classical Fisher-KPP problem. 
It is required that the nonlinear function $1 - \e^2F'(\rho)$ remains positive, which reads exactly $\e^2 F'(0)<1$ in our context.
We present below a direct proof based on the method of sub- and supersolutions, following the method developed by Berestycki and Hamel in \cite{BerestyckiHamel}.

\subsubsection{The linearized problem}

%We first investigate the linearization of equation (\ref{eqtw}) in the neighbourhood of the steady state $\U \equiv 0$, which reads
%\begin{equation} \label{eq-lin}
% (\e^2 s^2-1^2)\U'' +(\e^2 F'(0)-1)s \U'  =  F'(0)\U \hbox{ in } \R.
%\end{equation}

%
%\begin{prop} \label{prop-lin}
% Assume that $\e^2 F'(0)< 1$ and $s\in [s^*,\frac{1}{\e})$. Then there exists positive constants $\lambda_s \leq \Lambda_s$ such that $\U (x)= e^{-\lambda x}$ is a solution of (\ref{eq-lin}) 
%if and only if $\lambda=\lambda_s$ or $\Lambda_s$.
%\end{prop}

%{\bf Proof.} A function $\U (x)= e^{-\lambda x}$ satisfies (\ref{eq-lin}) if and only if 
%$$ P_s(\lambda):=(\e^2 s^2-1)\lambda^2 +(1-\e^2 F'(0))s\lambda  - F'(0)=0.$$
%The second order polynomial equation admits two solutions if and only if 
%$$(1-\e^2 F'(0))^2 s^2 +4 F'(0) (\e^2 s^2-1)=(1+\e^2 F'(0))^2 s^2 -4 F'(0)>0,$$
%which is true since $s>s^*$. As $P_s(0)=-F'(0)<0$ and $P_s(+\infty)=-\infty$ since $s<\frac{1}{\e}$, these two solutions have the same sign. Lastly, 
%$P_s'(0)=(1-\e^2 F'(0))s \geq 0$ and thus these two zeros are positive. $\Box$

\begin{prop} \label{prop-supersol}
 Let $\lambda_s$ be the smallest (positive) root of the characteristic polynomial \eqref{eq:charac}. Then 
$\overline{\U}(z)=\min \{1, e^{-\lambda_s z}\}$ is a supersolution  of (\ref{eqtw}).
\end{prop}

\begin{proof} 
Let $r(z)=e^{-\lambda_s z}$. Then as $r$ is decreasing and $F$ is concave, it is easy to see that $r$ is a supersolution of (\ref{eqtw}). 
On the other hand, the constant function $1$ is clearly a solution of (\ref{eqtw}). 
We conclude since the minimum of two supersolutions is a supersolution. 
\end{proof}

\subsubsection{Resolution of the problem on a bounded interval}

\begin{prop}\label{prop-pbma}
 For all $a>0$ and $\tau\in \R$, there exists a solution $\U_{a,\tau}$ of 
\begin{equation}\label{eq-pbma}
 \left\{ \begin{array}{l}
          (\e^2 s^2-1)\U_{a,\tau}'' +(\e^2 F'(\U_{a,\tau})-1)s \U_{a,\tau}' =  F(\U_{a,\tau}) \hbox{ in } (-a,a),\\
\U_{a,\tau}(-a)=\overline{\U} (-a+\tau),\\
\U_{a,\tau}(a)=\overline{\U} (a+\tau).\\ 
         \end{array}\right.
\end{equation}
Moreover, this function is nonincreasing over $(-a,a)$ and it is unique in the class of nonincreasing functions.
\end{prop}
In order to prove this result, we consider the following sequence of problems:
\begin{itemize} 
\item $\U_0(z)=\overline{\U}(z+\tau)$
\item $\U_{n+1}$ is solution to 
\begin{equation}\label{eq-rhon}
\begin{cases}
          (\e^2 s^2-1)\U_{n+1}'' +(\e^2 F'(\U_n)-1)s \U_{n+1}'+M\U_{n+1} =  F(\U_n) +M\U_n \hbox{ in } (-a,a),\\
\U_{n+1}(-a)=\overline{\U} (-a+\tau),\\
\U_{n+1}(a)=\overline{\U} (a+\tau),\\ 
\end{cases}
\end{equation}
where $\overline{\U}$ is defined in Proposition \ref{prop-supersol} and $M>\frac{s^2}{2} \big( \e^2 F'(0)-1\big)$ 
is large enough so that $s\mapsto F(s)+Ms$ is increasing.
\end{itemize}

\begin{lem} \label{lem-subsuper}
 The sequence $(\U_n)_n$ is well-defined. The functions $z\mapsto \U_n(z)$ are nonincreasing and for all $z\in (-a,a)$, the sequence $(\U_n(z))_n$ is nonincreasing. 
\end{lem}

\begin{proof} We prove this Lemma by induction.
Clearly, $\U_0$ is nonincreasing. 
First, one can find a unique weak solution $\nu_1 \in \mathcal{C}^0 ([-a,a])$ of
\begin{equation}
\begin{cases}
          (\e^2 s^2-1)\U_{1}'' +(\e^2 F'(\U_0)-1)s \U_{1}'+M\U_{1} =  F(\U_0) +M\U_0 \hbox{ in } (-a,a),\\
\U_{1}(-a)=\overline{\U} (-a+\tau),\\
\U_{1}(a)=\overline{\U} (a+\tau),\\ 
\end{cases}
\end{equation}
using the Lax-Milgram theorem and noticing that the underlying operator is coercive since $M>\frac{s^2}{2} \big( \e^2 F'(0)-1\big)$ and $s<\e^{-1}$.

Let $w_0=\U_1-\U_0$. As $\U_0$ is a supersolution of equation (\ref{eqtw}), one has 
\begin{equation*}
 \left\{ \begin{array}{l}
          (\e^2 s^2-1)w_0'' +(\e^2 F'(\U_0)-1)s w_0'+Mw_0 \leq  0 \hbox{ in } (-a,a),\\
w_0(-a)=w_0(a)=0.\\
         \end{array}\right.
\end{equation*}
As $M>0$, the weak maximum principle gives $w_0\leq 0$, that is, $\U_1\leq \U_0$. 

Define the constant function $\underline{\U}=\overline{\U} (a+\tau)$. It satisfies
$$(\e^2 s^2-1)\underline{\U}'' +(\e^2 F'(\U_0)-1)s \underline{\U}'+M\underline{\U} = M\underline{\U}\leq F(\underline{\U})+M\underline{\U}\leq 
F(\U_0) +M\U_0$$
in $(-a,a)$ since $s\mapsto F(s)+Ms$ is increasing and $\U_0 (z)= \overline{\U} (z+\tau) \geq \overline{\U} (a+\tau) = \underline{\U}$ by monotonicity of $\overline{\U}$. 
The same arguments as above lead to $\U_1 \geq \underline{\U}$.

Assume that Lemma \ref{lem-subsuper} is true up to rank $n$. The existence and the uniqueness of $\U_{n+1}$ follow from the same arguments as that of $\U_1$. 
Let $w_n=\U_{n+1}-\U_n$.
As $F$ is concave and $\U_{n-1}\geq \U_n$, we know that $F'(\U_{n-1})\leq F'(\U_n)$. As $\U_n$ is nonincreasing,
we thus get
\begin{equation*}
 \left\{ \begin{array}{l}
          (\e^2 s^2-1)w_n'' +(\e^2 F'(\U_n)-1)s w_n'+Mw_n \leq  0 \hbox{ in } (-a,a),\\
w_n(-a)=w_n(a)=0.\\
         \end{array}\right.
\end{equation*}
Hence, $w_n\leq 0$ and thus $\U_{n+1}\leq \U_n$. Similarly, one easily proves that $\U_{n+1}\geq \underline{\U}$ in $(-a,a)$.

%Assume first that $F$ is of class $\mathcal{C}^2 ([0,1])$. 
Differentiating (\ref{eq-rhon}) and denoting $v=\U_{n+1}'$, one gets 
$$(\e^2 s^2-1)v'' +\big(\e^2 F'(\U_0)-1\big)s v'+\big(M+\e^2 F''(\U_0) \U_0'\big)v = \big( F'(\U_0) +M\big)\U_0'\leq 0 \hbox{ in } (-a,a)$$
since $s\mapsto F(s)+Ms$ is increasing and $\U_0$ is nonincreasing. As $F$ is concave, the zeroth-order term is positive and thus the elliptic maximum principle
ensures that $v$ reaches its maximum at $z=-a$ or at $z=a$. But as 
$\overline{\U} (a+\tau) \leq \U_{n+1} (z) \leq \overline{\U} (z+\tau)$ for all $z\in (-a,a)$, one has 
$$v(-a) \leq \limsup_{z\to -a^+} \dfrac{\U_{n+1} (z) -\U_{n+1}(-a)}{z+a} \leq \limsup_{z\to -a^+} \dfrac{\overline{\U} (z+\tau) -\overline{\U}(-a+\tau)}{z+a}\leq 0$$
and similarly $v(a)\leq 0$. Thus $v\leq 0$, meaning that $\U_{n+1}$ is nonincreasing.
%
%If $F$ is just $\mathcal{C}^1$, then the monotonicity of $\U_{n+1}$ easily follows by approximation since $\U_{n+1}$ is uniquely defined.
\end{proof}

\begin{proof}[Proof of Proposition \ref{prop-pbma}.] As the sequence $(\U_n)_n$ is decreasing and bounded from below, it admits a limit $\U_{a,\tau}$ as $n\to +\infty$. 
It easily follows from the classical regularity estimates that $\U_{a,\tau}$ satisfies the properties of Proposition \ref{prop-pbma}.

If $\U_1$ and $\U_2$ are two nondecreasing solutions of (\ref{eq-pbma}), then the same arguments as before give that  $\U_1^\mu<\U_1$
in $\Sigma_\mu$ for all $\mu\in (0,2a)$. Hence, $\U_1\leq \U_2$ and a symmetry argument gives $\U_1\equiv \U_2$. 
\end{proof}

\begin{lem}
For all $a>0$, there exists $\tau_a\in\R$ such that $\U_{a,\tau_a}(0)=\frac{1}{2}$.
\end{lem}

\begin{proof} Define $I(\tau):= \U_{a,\tau}(0)$. It follows from the classical regularity estimates and from the uniqueness of $\U_{a,\tau}$ that 
$I$ is a continuous function. Moreover, as $\U_{a,\tau}$ is nonincreasing, one has
$$\overline{\U}(a+\tau)\leq I(\tau)\leq \overline{\U}(-a+\tau),$$
where $\overline{\U}$ is defined in Proposition \ref{prop-supersol}.
As $\overline{\U}(\cdot +\tau)\to 0$ as $\tau\to +\infty$ and $\overline{\U}(\cdot +\tau)\to 1$ as $\tau\to -\infty$ locally uniformly on $\R$, one has 
$I(-\infty)=1$ and $I(+\infty)=0$. The conclusion follows.  
\end{proof}

\subsubsection{Existence of traveling fronts with speeds $s\in [s^*(\eps),  \e^{-1})$}

We conclude by giving the proof of Theorem \ref{thm-diffusion} as a combination of the above results.

\begin{proof}[Proof of Theorem \ref{thm-diffusion}.] 
Consider a sequence $(a_n)_n$ such that $\lim_{n\to +\infty} a_n=+\infty$ and define $\U_n (z):= \U_{a_n,\tau_{a_n}}$ for all $z\in [-a_n,a_n]$. 
This function is decreasing and satisfies $\U_n (0)=1/2$, $0 \leq \U_n \leq 1$ and 
$$(\e^2 s^2-1)\U_n'' +(\e^2 F'(\U_{n})-1)s \U_n' =  F(\U_{n}) \hbox{ in } (-a_n,a_n).$$
Similar arguments as in the proof of Proposition \ref{prop-nonex} yield that the sequence
$(\U_{a_n,\tau_{a_n}})_n$ converges in $\mathcal{C}^0_{loc}(\R)$
as $n\to +\infty$ to a function $\nu$, up to extraction. Then $\U$ satisfies
$$(\e^2 s^2-1)\U'' +(\e^2 F'(\U)-1)s \U' =  F(\U) \hbox{ in } \R,$$ 
it is nonincreasing, $0\leq \U \leq 1$ and $\U(0)=1/2$.

Define $\ell_\pm := \lim_{z\to \pm\infty} \nu (z)$. Passing to the (weak) limit in the equation satisfied by $\U$, one gets $F(\ell_\pm)=0$. As $0\leq \ell_\pm\leq 1$, 
the hypotheses on $F$ give $\ell_\pm \in \{0,1\}$. On the other hand, as $\U$ is nonincreasing, one has 
$$\ell_+ \leq \U (0)=1/2 \leq \ell_-.$$
We conclude that 
$\ell_-=\U(-\infty)=1$ and $\ell_+=\U(+\infty)=0$. 
\end{proof}

The following classical inequality satisfied by the traveling profile will be required later. 
\begin{lem}\label{lem:varphi<lambda}
The traveling profile $\nu$ satisfies: $\forall z \; \nu'(z) + \lambda \nu(z)\geq 0$, where $\lambda $ is the smallest positive root of \eqref{eq:charac}.
\end{lem}
\begin{proof}
We introduce $\varphi(z) = -\frac{\nu'(z)}{\nu(z)}$. It is nonnegative, and it satisfies the following first-order ODE with a source term
\[ \left( \e^2s^2 - 1 \right) \left( - \varphi'(z) + \varphi(z)^2  \right) + \left(1 - \e^2 F'(\nu(z)) \right)s \varphi(z) = \dfrac{F(\nu(z))}{\nu(z)} \, . \]
Since $F$ is concave, $\varphi$ satisfies the differential inequality
\[ \left(1 - \e^2s^2 \right) \varphi'(z) \leq  \left(1 - \e^2s^2 \right) \varphi(z)^2 - \left(1 - \e^2 F'(0) \right)s \varphi(z) + F'(0)\, .  \]
The right-hand-side is the characteristic polynomial of the linearized equation \eqref{eq:charac}. Moreover the function $\varphi$ verifies $\lim_{z\to -\infty} \varphi(z) = 0$. Hence a simple ODE argument shows that $\forall z \; \varphi(z) \leq \lambda$.
\end{proof}

%\subsection*{Non-existence of continuous waves when $\e^2 F'(0)>1$}

%Consider $F(\U)=\U (1-\U)$ and fix $\e>1$. Assume that there exists a traveling front $\U$ of speed $\frac{1}{\e}$. Then this function satisfies
%$$\U' = \frac{\alpha \U (1-\U)}{1-2\alpha \U} \hbox{ where } \alpha=\frac{\e^2}{\e^2-1}>0.$$
%One easily checks that 
%$$(\ln \U)'-(1-2\alpha) (\ln (1-\U))'=\alpha.$$
%Hence, there exists some constant $C\in\R$ such that
%\begin{equation} \label{eq-cex} \U(x) (1-\U(x) )^{2\alpha -1} = \alpha x + C \hbox{ for all } x\in\R.\end{equation}
%As $\e>1$, one has $2\alpha >1$. Thus, (\ref{eq-cex}) is not compatible with the boundary conditions $\U(-\infty)=1$ and $\U(+\infty)=0$. $\Box$

%%%%%%%%%%%%%%%%%%%%%%%%%%%%%%%%%%%%%%%%%%%%%%%%%%%%%%%%%%

\subsection{Proof of Theorem \ref{thm-diffusion}.(c): Existence of weak traveling fronts of speed $s=\e^{-1}$ in the parabolic regime}
The aim of this Section is to prove that in the parabolic regime $\e^2 F'(0)<1$, there still exists traveling fronts in the limit case $s=\e^{-1}$ 
but in the weak sense. 

\begin{prop} \label{prop-parab-critique}
 Assume that $\e^2 F'(0)<1$. Then there exists a weak traveling front of speed $s=\e^{-1}$. 
\end{prop}

\begin{proof}
Let $s_n = \e^{-1}-1/n$ for all $n$ large enough so that $s_n \geq s^*(\e)$. We know from the previous Section that we can associate with the speed $s_n$ 
a smooth traveling front $\U_n$ and that we can assume, up to translation, that $\U_n (0)=1/2$. Multiplying equation (\ref{eqtw}) by $\U_n'$ and integrating 
by parts over $\R$, one gets 
\begin{align*} s_n\big(1-\e^2 F'(0)\big)\int_\R  \U_n'(z)^2 dz &\leq  s_n\int_\R \Big(1-\e^2 F'\big(\U_n (z)\big)\Big) \U_n'(z)^2 dz \\
  & = -\int_\R F\big(\U_n (z) \big) \U_n '(z) dz \\
&=-\int_0^1 F(u)du.\\
  \end{align*}
Hence, as $\e^2 F'(0)<1$, the sequence $(\U_n')_n$ is bounded in $L^2 (\R)$ and one can assume, up to extraction, that it admits a weak limit $V_\infty$ in $L^2 (\R)$.
It follows that the sequence $(\U_n)_n$ converges locally uniformly to $\nu_\infty (z):= \int_0^z V_\infty (z')dz' +1/2$. 
Passing to the limit in (\ref{eqtw}), we get that this function is a weak solution of 
$$ - \left(1 - \e^2 F'(\U_\infty(z)) \right) s \U_\infty'(z) =  F(\U_\infty(z))\, , \quad z\in  \R\, ,$$
which ends the proof. 
\end{proof}

%%%%%%%%%%%%%%%%%%%%%%%%%%%%%%%%%%%%%%%%%%%%%%%%%%%%

\subsection{Proof of Theorem \ref{thm-hyperbolic}.(b): Existence of weak traveling fronts of speed $s = \e^{-1}$ in the hyperbolic regime}

\label{ssec:weakTW}

In this Section we investigate the existence of traveling fronts with critical speed $s=\e^{-1}$ in the hyperbolic regime $\e^2 F'(0)=1$.

\begin{proof}[Proof of Theorem \ref{thm-hyperbolic}.] 

The function $G(\rho) := \e^2 F(\rho) - \rho$ is concave, and vanishes when $\rho = 0$. Furthermore, $G(1) < 0$ and $G'( 0 ) = \e^2 F'(0) - 1 \geq 0$. We now distinguish between the two cases $\e^2 F'(0)>1$ and $\e^2 F'(0)=1$. 

\begin{enumerate}
\item \textbf{First case: $\e^2 F'(0)>1$.} As $G'$ is decreasing, there exists a unique $\theta_\e \in (0,1)$ such that $G$ vanishes.
\item \textbf{Second case: $\e^2 F'(0)=1$.} The only root of $G$ is $\rho = 0$. In this case we set $\theta_{\e} = 0$.
\end{enumerate}

\noindent For both cases, we have $G'(\rho)<0$ for all $\rho>\theta_\eps$ since $G$ is strictly concave and $G(0) = G(\theta_\eps) =  0$.
%
%take any $r\in (0,\theta_\eps)$, one has $\e^2 F(r) > r$. By concavity of $F$, we get 
%$$\forall \rho \geq \theta_\eps, \quad F'(\rho)\leq \dfrac{F(\theta_\eps)-F(r)}{\theta_\eps-r}<\dfrac{\e^{-2}(\theta_\eps -r)}{\theta_\eps-r}=\e^{-2}.$$
Hence, $\e^2 F'(\rho)<1$ for all $\rho>\theta_\eps$. Set $\nu$ the maximal solution of 
\begin{equation} \label{eq-ODE}\ \begin{cases}
                                        \nu'(z)=\displaystyle \frac{\e F(\nu(z))}{\e^2 F'(\nu(z))-1},\medskip \\
\nu(0)=\dfrac{1+\theta_\eps}{2}> \theta_\eps.
                                       \end{cases}
\end{equation}
Let $I$ be the (maximal) interval of definition of $\nu$, with $0\in I$, 
%Assume that $\rho$ is defined on a maximal interval $I$ with $0\in I$. 
and $$z_0=\sup \{z\in I, \nu(z)>\theta_\eps\}.$$
%It is easy to see that $I=(-\infty,s_0)$ and }

\subsubsection*{1- Conclusion of the argument in the first case: $\e^2 F'(0)>1$.}

Since $\theta_\e >0$, we have necessarily $z_0 < + \infty$. From \eqref{eq-ODE}, $\nu$ is decreasing on $(-\infty,z_0)$. Thus, we have $\nu(z)\to \theta_\eps$ as $z\to z_0^-$. Moreover, one easily gets $\nu(-\infty)=1$. 

We set $\nu(z_0) = \theta_\eps$ and we extend $\nu$ by $0$ over $(z_0,\infty)$. We observe that $\nu$ is a weak solution, in the sense of distributions, of 
$$\left(\e^2 F(\nu) -\nu\right)'=\e F(\nu) \hbox{ on } \R$$
since $\e^2 F(0)=0$ and $\e^2 F(\theta_\eps)=\theta_\eps$. 

Up to space shifting $z - z_0$, we may assume that the discontinuity arises at $z = 0$.

\paragraph{Example: the case $F(\rho) = \rho(1 - \rho)$ and $\eps>1$.} The traveling profile solves
\[ \U'(z) = \dfrac{\eps \U(z) (1 - \U(z))}{\eps^2 - 1 - 2\eps^2\U(z)}\, , \]
or equivalently
\[ \U(z)^{\eps^2 - 1}\left( 1 - \U(z)\right)^{\eps^2 + 1} = k e^{\eps z}\, . \]
The constant $k$ is determined by the condition $\U(0) = \theta_\eps = 1 - \eps^{-2}$. Finally the traveling profile $\nu(z)$ satisfies the following implicit relation:
\begin{equation} 
\U(z)^{\eps^2 - 1}\left( 1 - \U(z)\right)^{1 + \eps^2} = \left(1 - \eps^{-2} \right)^{\eps^2 - 1}\left(\eps^{-2} \right)^{\eps^2 + 1} e^{\eps z} = \left( \eps^2 - 1 \right)^{\eps^2 - 1}  e^{\eps z + 2\eps^2 \log \eps^2}\, . \label{eq:nu cas eps>1}
\end{equation}

\subsubsection*{2- Conclusion of the argument in the second case: $\e^2 F'(0)=1$.}

The difference here is that $\theta_{\e} = 0$. To conclude the proof as previously, we just need to check that $z_0$ is finite. We argue by contradiction. Assume $z_0 = +\infty$. Linearizing the r.h.s. of \eqref{eq-ODE} near $\nu = 0$, we get
\begin{equation}
\nu'(z) = \frac{F'(0)}{\e F''(0)} + o(\nu(z))\,,\quad \text{as $z\to +\infty$}
\end{equation}
We get a contradiction because $\e^{-1} F'(0)/F''(0) <0$.

Finally, we create a continuous front with the same extension idea as for the first case.
%It is easy to see that $\e^2\in (1/F'(0),+\infty)\mapsto \theta_\eps$ is continuous and nonincreasing. 
%Hence, it admits a limit $\theta$ when \red{$\e\to 1/\sqrt{F'(0)}$}.
%We consider the family of weak traveling front solutions $\nu_\e$ of (\ref{eq-ODE}). 
%Then, the family $(\nu_\e)_\e$ converges as $\e\to 1/\sqrt{F'(0)}$ to a measurable nonincreasing function $\nu$ which satisfies
%$$\left(\frac{1}{F'(0)} F(\nu) -\nu\right)'=\frac{1}{\sqrt{F'(0)}} F(\nu) \hbox{ on } \R$$
%in the sense of distributions. This is the desired profile.
% and $\nu(0)=\frac{1+\theta}{2}$. 
%The same arguments as before give the good limits as $x\to\pm\infty$. 

\paragraph{Example:  the case $F(\rho) = \rho(1 - \rho)$ and $\eps=1$.} The traveling profile reads \eqref{eq:nu cas eps>1}:
\begin{equation*} 
\nu(z) = \left(1 - e^{z/2} \right)_+\, .
%%\label{eq:nu cas eps=1}
\end{equation*}
\end{proof}

\subsection{Proof of Theorem \ref{thm-diffusion}.(d) and Theorem \ref{thm-hyperbolic}.(c): Existence of supersonic traveling fronts $s>\e^{-1}$}
\label{sec:supersonic}

In this Section we investigate the existence of supersonic traveling fronts with speeds above the maximal speed of propagation $s> \e ^{-1}$. These fronts are essentially driven by growth. The existence of such "unrealistic" fronts is motivated by the extreme case $\e \to +\infty$ for which we have formally $\partial_t \rho = F(\rho)$ \eqref{eq:telegraph intro}. There exist traveling fronts of arbitrary speed which are solutions to $-s \U' = F(\U)$.

\begin{prop} Given any speed $s > \e ^{-1}$ there exists a smooth traveling front $\nu(x - st)$  with this speed.
\end{prop}

\begin{proof} We sketch the proof. We give below the key arguments derived from phase plane analysis. The same procedure as developped in Section \ref{ssec:sub-sup} based on sub- and supersolutions could be reproduced based on the following ingredients.

\begin{figure}
\begin{center} 
\includegraphics[width = 0.6\linewidth]{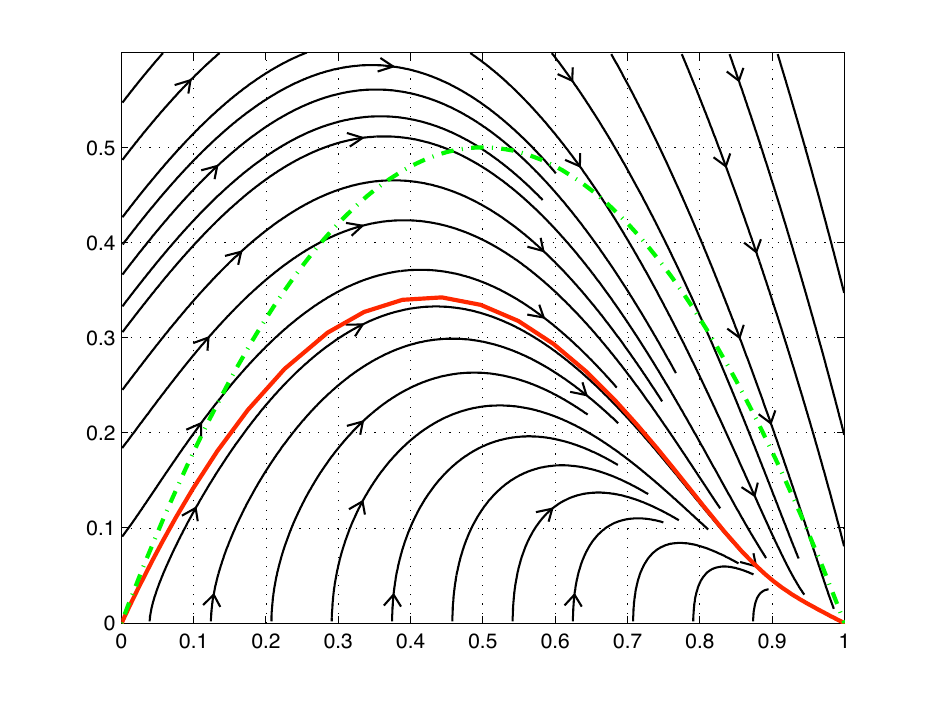}
\end{center}
\caption{Supersonic traveling front in the phase plane $(V,V')$ for the nonlinearity $F(\rho) = \rho(1-\rho)$, and parameters $\e  = \sqrt{2}$, $s = 1>\e ^{-1}$. Be aware of the time reversal $\U(z) = V(-z)$, which is the reason why $V'\geq0$. The red line represents the traveling profile, and the green line represents the supersolution $k F(v)$.}
\label{fig:proof supersonic}
\end{figure}

We learn from simple phase plane considerations associated to \eqref{eqtw} that the situation is reversed in comparison to the classical Fisher-KPP case (or  $\e^2 F'(0)<1$ and $s\in [ s^*(\e),\e^{-1})$). Namely the point $(0,0)$ is a saddle point (instead of a stable node) whereas $(1,0)$ is an unstable node (instead of saddle point). This motivates "time reversal": $V(z) = \nu(-z)$. Equation \eqref{eqtw} becomes 
\begin{equation*} \label{eqtwrev}
 (\e^2 s^2 - 1)V''(z) - \left(\e^2 F'\left(V(z)\right)-1\right)s V'(z) =  F\left(V(z)\right) \, , \quad z\in  \R\, .
\end{equation*}
We make the classical phase-plane transformation $V' = P$ \cite{KPP,Fife.McLeod}. 
%We multiply \eqref{eqtwrev} by $dz/dV$
We end up with the implicit ODE with Dirichlet boundary conditions for $P$:
\begin{equation*} 
(\e^2 s^2 - 1) P'(v) - \left(\e^2 F'\left( v \right)-1\right)s = \dfrac{F(v)}{P(v)}\, , \quad P(0) = P(1) = 0\, .
\end{equation*}
The unstable direction is given by $P(v) = \lambda v$ where $\lambda$ is the positive root of 
\begin{equation} 
(\e^2 s^2 - 1) \lambda - \left(\e^2 F'\left( 0 \right)-1\right)s = \dfrac{F'(0)}{\lambda}\, . 
\label{eq:caracrev}
\end{equation}
Since $F$ is concave we deduce that $P(v) = \lambda v$ is a supersolution as in Proposition \ref{prop-supersol}. In fact, denoting $Q(v) = P(v) - \lambda v$ we have
\begin{align*}
%\begin{array}{lcl}  
(\e^2 s^2 - 1) Q'(v) =  (\e^2 s^2 - 1)(P'(v) - \lambda) & \leq  s \e^2 \left( F'(v) - F'(0) \right) + \frac{F(v)}{P(v)} - \frac{F'(0)}{\lambda}, \\
& \leq  F'(0) v \left( \frac{1}{P(v)} - \frac{1}{\lambda v} \right) \\
& \leq  - \frac{F'(0)}{\lambda P(v)} Q(v)\, . 
%\end{array}
\end{align*}
Hence the trajectory leaving the saddle point $(0,0)$ in the phase plane $(V,V')$ remains below the line $V'\leq \lambda V$. 

On the other hand it is straightforward to check that $k F(v)$ is a supersolution where $k = \e^2 s /(\e^2 s^2 - 1)$. We denote $R(v) = P(v) - k F( v)$. We have $ks > 1$ and 
\begin{align*}
(\e^2 s^2 - 1) R'(v) = (\e^2 s^2 - 1)(P'(v) - k F'(v)) & = 
\e^2 s F'(v) - s + \dfrac {F(v)}{P(v)} - (\e^2 s^2 - 1) k F'(v) \\ 
& = - s + \dfrac 1 k - \dfrac{R(v)}{ kP(v)} < - \dfrac{R(v)}{ kP(v)}\, .  
\end{align*}
We also show that initially (as $v\to 0$) we have $kF'(0) > \lambda$. This proves that $R(v)\leq 0$ for all $v\in (0,1)$. Indeed, we plug $kF'(0)$ in place of $\lambda$ into \eqref{eq:caracrev} and we get
\[ (\e^2 s^2 - 1) kF'(0) - \left(\e^2 F'\left( 0 \right)-1\right)s = s > \dfrac1k =  \dfrac{F'(0)}{kF'(0)}\, . \]

As a conclusion the trajectory leaving the saddle node at $(0,0)$ is trapped in the set $\{ 0\leq v\leq 1\, , \, 0\leq  p \leq kF(v)  \}$ (see Fig. \ref{fig:proof supersonic}). By the Poincar\'e-Bendixon Theorem it necessarily converges to the stable node at $(1,0)$. This heteroclinic trajectory is the traveling front in the supersonic case.
\end{proof}

\section{Linear stability of traveling front solutions}

\label{sec:linstab}

In this Section we investigate the linear stability of the traveling front having minimal speed $s = s^*(\eps)$ in both the parabolic and the hyperbolic regime. We seek stability in some weighted $L^2$ space. The important matter here is to identify the weight $e^\phi$. The same weight shall be used crucially for the nonlinear stability analysis (Section \ref{sec:NLstab}). 

We recall that the minimal speed is given by
\[ s^*(\eps) = \begin{cases} \dfrac{2\sqrt{F'(0)}}{1 + \e^2 F'(0)} & \text{if $\e^2 F'(0)<1$} \\
\e^{-1} & \text{if $\e^2 F'(0)\geq 1$}\end{cases} 
\]
The profile of the wave has the following properties in the case $\e^2 F'(0)<1$:
\[ \forall z \quad  \nu(z) \geq 0\, , \quad \partial_z \nu(z) \leq 0\, , \quad \partial_z \nu(z) + \lambda \nu(z)\geq 0\, ,  \]
where the decay exponent $\lambda$ is 
\[ \lambda = \dfrac{s(1 - \e^2 F'(0))}{2(1 - \e^2 s^2 )} = \dfrac{1 + \e^2 F'(0)}{1 - \e^2 F'(0)}\, . \]

We will use in this Section the formulation \eqref{eq:telegraph intro} of our system.
%\begin{equation}
%\label{eq:telegraph}
%\left\{
%\begin{array}{l}
%\partial_t \rho + \partial_x \left(\dfrac j\e\right) = F(\rho) \\
%\e \partial_t j + \partial_x \rho = -\dfrac j\e\, . 
%\end{array}
%\right.
%\end{equation}
%
The linearized system around the stationary profile $\U$ in the moving frame $z = x - st$ reads 
\begin{equation}
\label{eq:telegraph}
\left\{
\begin{array}{l}
(\partial_t - s\partial_z) u + \partial_{z} \left(\dfrac v\e\right) = F'(\nu) u \\
\e (\partial_t - s\partial_z) v + \partial_{z} u = -\dfrac v\e\, . 
\end{array}
\right.
\end{equation}

\begin{prop}
Let $\e>0$. In the hyperbolic regime $\eps^2 F'(0)\geq 1$ assume in addition that the initial perturbation has the same support as the wave. There exists a function $\phi_\e(z)$ such that the minimal speed traveling front is linearly stable in the weighted $L^2(e^{2\phi_\eps(z)}dz)$ space. More precisely the following Lyapunov identity holds true for solutions of the linear system \eqref{eq:telegraph},
\begin{equation*}
\dfrac d{dt} \left( \dfrac12 \int_\R \left( |u|^2 + |v|^2\right) e^{2\phi_\eps(z)}\, dz \right) \leq 0\, . 
\end{equation*}
\end{prop}

\begin{proof} We denote $\phi = \phi_\eps$ for the sake of clarity. 
We multiply the first equation by $ue^{2\phi} $, and the second equation by $ve^{2\phi}$, where $\phi$ is  to be determined. We get
\begin{align*} 
&\dfrac d{dt} \left( \dfrac12 \int_\R |u|^2 e^{2\phi(z)}\, dz\right) + \dfrac s2 \int_\R |u|^2 \partial_z e^{2\phi(z)} \, dz + \int_\R \partial_z \left(\dfrac v\e\right) u e^{2\phi(z)}\, dz = \int_\R F'(\nu) |u|^2 e^{2\phi(z)}\, dz\, ,\\
&\dfrac d{dt} \left( \dfrac12 \int_\R |v|^2 e^{2\phi(z)}\, dz\right) + \dfrac s2 \int_\R |v|^2 \partial_z e^{2\phi(z)} \, dz + \int_\R \partial_z \left(\dfrac u\e\right) v e^{2\phi(z)}\, dz = - \dfrac1{\e^2} \int_\R  |v|^2 e^{2\phi(z)}\, dz\, .  
\end{align*}
Summing the two estimates we obtain
\begin{multline*}
\dfrac d{dt} \left( \dfrac12 \int_\R \left( |u|^2 + |v|^2\right) e^{2\phi(z)}\, dz \right) \\ +   \int_\R \left( s \partial_z\phi(z)  - F'(\nu) \right) |u|^2 e^{2\phi(z)} \, dz + \int_\R \left(   s \partial_z\phi(z) + \dfrac1{\e^2} \right)|v|^2 e^{2\phi(z)}\, dz - \dfrac2\e \int_\R \left(  \partial_z \phi(z)\right) u v  e^{2\phi(z)}\, dz = 0\, .
\end{multline*}

We seek an energy dissipation estimate, see \eqref{eq:Lyap1} below. Therefore we require that the last quadratic form acting on $(u,v)$ is nonnegative. This is guaranteed if $  \partial_z \phi \geq 0 $ and the following discriminant is nonpositive:
\begin{align} 
\Delta(z) &= \dfrac4{\e^2}\left(\partial_z \phi(z)\right)^2 - 4 \left(   s \partial_z\phi(z)  - F'(\nu) \right)  \left(   s \partial_z\phi(z) + \dfrac1{\e^2} \right) \nonumber \\
& = \dfrac4{\e^2}\left( \left(1 - \e^2 s^2\right) \left(\partial_z \phi(z)\right)^2  
-   s \left( 1 - \e^2 F'(\nu)\right) \partial_z \phi(z)  +   F'(\nu) \right)  \, . \label{eq:discr}
\end{align}
The rest of the proof is devoted to finding such a weight $\phi(z)$ satisfying this sign condition. We distinguish between the parabolic and the hyperbolic regime.

\subsubsection*{1- The parabolic regime.}
In the case $\e^2 F'(0) < 1 $ we have $\e^2s^2 <  1$. Hence the optimal choice for $\partial_z \phi$ is: 
\begin{equation*} \partial_z \phi(z)  = \dfrac{s\left( 1 - \e^2 F'(\nu) \right)}{2\left(1 - \e^2 s^2\right)} =  \lambda \frac{1 - \e^2 F'(\nu) } { 1 - \e^2 F'(0) } \geq 0\, .  
\end{equation*} 
Notice that $\partial_z\phi \to  \lambda$ as $z\to +\infty$.
We check that the discriminant is indeed nonpositive:
\begin{align*} 
\e ^2 \Delta(z) & =  - 4 \left(1 - \e^2 s^2\right) \left(\partial_z \phi(z)\right)^2 + 4 F'(\nu)   \\
& =  \dfrac{ 1}{ (1 - \e^2 s^2)} \left( - s^2 \left(1 + \e^2F'(\nu)\right)^2  + 4 F'(\nu) \right) \\
& = \dfrac{1}{\left(1 - \e^2 F'(0)\right)^2} 
\left( - 4F'(0)  \left(1 + \e^2F'(\nu)\right)^2 + 4 F'(\nu) \left(1 + \e^2 F'(0)\right)^2 \right) \\
&= \dfrac{- 4}{\left(1 - \e^2 F'(0)\right)^2} \left( F'(0) - F'(\nu)\right) \left( 1 - \e^4 F'(0) F'(\nu)\right)\, .
\end{align*}
We have $\Delta(z) \leq 0$ since $\forall z\; F'(\nu(z)) \leq F'(0) $ and $\e^2 F'(0)<1$. 
Since the quadratic form is nonnegative, we may control it by a sum of squares. This is the purpose of the next computation. We have 
\begin{multline*}
\left|\dfrac2\e \int_\R \left(\partial_z \phi(z)\right) u v  e^{2\phi(z)}\, dz\right| \\ \leq \int_\R \left(  s \partial_z\phi(z)  - F'(\nu) - A(z)\right) |u|^2 e^{2\phi(z)} \, dz + \int_\R \left(  s \partial_z\phi(z) + \dfrac1{\e^2} - A(z)\right)|v|^2 e^{2\phi(z)}\, dz\, ,
\end{multline*}
where $A(z)$ is solution of 
\[ 4 \left(   s \partial_z\phi(z)  - F'(\nu) - A(z)\right) \left(   s \partial_z\phi(z) + \dfrac1{\e^2} - A(z)\right)= \dfrac4{\e^2}\left(\partial_z \phi(z)\right)^2\, .  \]
A straightforward computation gives
\begin{align*} 
2 A(z) &= \left( 2s\partial_z\phi(z)  - F'(\nu) + \dfrac 1{\e^2}\right) - \left(\left(  F'(\nu) + \dfrac 1{\e^2} \right)^2 + \dfrac4{\e^2}\left(\partial_z \phi(z)\right)^2\right)^{1/2} \\
& =  \dfrac{1 - \e^2 F'(\nu)}{\e^2\left(1 - \e^2 s^2\right)}- \dfrac1{\e^2\left(1 - \e^2 s^2\right)} \left( \left(1 - \e^2 s^2\right)^2 \left(1 + \e^2 F'(\nu)\right)^2 + \e^2 s^2 \left(1 - \e^2 F'(\nu)\right)^2  \right)^{1/2} \\
& = \dfrac{ 1 - \e^2 F'(\nu)}{\e^2\left(1 - \e^2 s^2\right)}\left( 1 - \left( \left( \dfrac{1 - \e^2 F'(0)}{1 + \e^2 F'(0)}\right)^4 \left( \dfrac{1 + \e^2 F'(\nu)}{1 - \e^2 F'(\nu)}\right)^2 + \dfrac{4\e^2 F'(0)}{\left(1 + \e^2 F'(0)\right)^2} \right)^{1/2} \right)  \\
& = \dfrac{ 1 - \e^2 F'(\nu)}{\e^2\left(1 - \e^2 s^2\right)}\left( 1 - \left( 1 + \left( \dfrac{1 - \e^2 F'(0)}{1 + \e^2 F'(0)}\right)^4 \left( \dfrac{1 + \e^2 F'(\nu)}{1 - \e^2 F'(\nu)}\right)^2 - \dfrac{\left(1 - \e^2 F'(0)\right)^2}{\left(1 + \e^2 F'(0)\right)^2} \right)^{1/2} \right)  \,.
\end{align*}
We clearly have $A(z)\geq 0$ since
\[ \forall z \quad \dfrac{1 + \e^2 F'(\nu)}{1 - \e^2 F'(\nu)} \leq \dfrac{1 + \e^2 F'(0)}{1 - \e^2 F'(0)}\, . \]
Finally we obtain in the case $\e^2 F'(0) < 1$,
\begin{equation} \label{eq:Lyap1} \dfrac d{dt} \left( \dfrac12 \int_\R \left( |u|^2 + |v|^2\right) e^{2\phi(z)}\, dz\right) + \int_\R A(z) \left( |u|^2 + |v|^2 \right) e^{2\phi(z)}\, dz \leq 0\, .  \end{equation}

\subsubsection*{2- The hyperbolic regime.}

We assume for simplicity that the support of the traveling profile is $\supp\U = (-\infty,0]$.\\
In the hyperbolic regime we have $s = \e^{-1}$, so the discriminant equation \eqref{eq:discr} reduces to 
\[ \Delta(z) = \dfrac4{\e^2}\left( -  s \left( 1 - \e^2 F'(\nu)\right) \partial_z \phi(z)  +   F'(\nu) \right)  \, .  \]
We naturally choose
\[ \partial_z \phi(z) = \dfrac{ \eps F'(\nu)}{1 - \eps^2 F'(\nu)}\, . \]
Within this choice for $\phi$ we get,
\begin{equation*}
\dfrac d{dt} \left( \dfrac12 \int_{z\leq0} \left( |u|^2 + |v|^2\right) e^{2\phi(z)}\, dz \right) + \int_{z\leq0} A(z) \left( \eps^2 F'(\nu(z)) u - v \right)^2 e^{2\phi(z)} \, dz =0\, ,
\end{equation*} 
where the additional weight in the dissipation writes:
\[ A(z) = \dfrac1{\eps^2\left( 1 - \eps^2 F'(\nu(z)) \right)} \, .\]

In the case $\eps^2 F'(0)>1$ we have $1 - \eps^2 F'(\nu(z))>0$ on $\supp\nu$ (see Section \ref{ssec:weakTW}). Notice that the monotonicity of $\phi$ may change on  $\supp\nu$ since $F'(\nu(z))$ may change sign. 
We observe that $A(z)$ is uniformly bounded from below on $\supp \U$.

In the transition case $\eps^2 F'(0) = 1 $, we have $\partial_z \phi(z) \to +\infty$ as $z\to 0^-$. We observe that $A(z)\to +\infty$ as $z\to 0^-$ too.

\end{proof}

\paragraph{Example: the case $F(\rho) = \rho(1 -\rho)$, and $\eps = 1$.} We can easily compute from Section~\ref{ssec:weakTW}
\[ \phi(z) = -\dfrac z2 -   \log\left( 1 - e^{z/2} \right) \, . \]
\begin{rmq}[Lack of coercivity] {\bf 1- The parabolic regime.} We directly observe that $A(z)\to 0$ as $z\to +\infty$ in the Lyapunov identity \eqref{eq:Lyap1}. This corresponds to the lack of coercivity of the linear operator. It has been clearly identified for the classical Fisher-KPP equation \cite{Gallay94,GallayRaugel}. This lack of coercivity is a source of complication for the next question, {\em i.e.} nonlinear stability (see Section \ref{sec:NLstab}). \\
{\bf 2- The hyperbolic regime.} The situation is more degenerated here: the dissipation  provides information about the relaxation of $v$ towards $\eps^2 F'(\U) u$ only.
\end{rmq}

\section{Nonlinear stability of traveling front solutions in the parabolic regime $\e^2 F'(0) < 1$}

\label{sec:NLstab}

In this Section we investigate the stability of the traveling profile having minimal speed $s = s^*(\e)$ in the parabolic regime. We seek stability in the energy class. Energy methods have been successfully applied to reaction-diffusion equations \cite{Gallay94,GallayRaugel,Risler,Gallay.Joly}. We follow the strategy developped in \cite{GallayRaugel} for a simpler equation, namely the damped hyperbolic Fisher-KPP equation. 

Before stating the theorem we give some useful notations. The perturbation is denoted by $ \w(t,z ) = \rho(t,z ) - \nu(z )  $ where $z  = x - s t$ is the space variable in the moving frame. We also need some weighted perturbation $\uu = e^{\phi} \w$, where $\phi$ is an explicit weight to be precised later \eqref{eq:weight phi}.

\begin{thm} \label{thm:NLstab}
For all $\e\in \left(0,1/\sqrt{F'(0)}\right)$ there exists a constant $c(\e)$ such that the following claim holds true: let $u^0$ be any compactly supported initial perturbation which satisfies 
\[ \| \w^0 \|_{H^1(\R)}^2 + \|\uu^0\|_{H^1(\R)}^2 \leq c(\e)\, ,  \]
then there exists $z_0 \in \R$ such that  
\[\sup_{t>0}\left( \|\partial_z  \w(t,\cdot)\|_2^2 + \int_{z  < z _0} |\w(t,z)|^2\, dz  + \|\uu(t,\cdot)\|_{H^1}^2\right) \leq c(\e)\, ,  \] 
remains uniformly small for all time $t>0$, and the perturbation is globally decaying in the following sense: 
\[ \left( \|\partial_z  \w\|_2^2  + \int_{z  < z _0} |\w|^2\, dz   + \|\partial_z  \uu\|_2^2 + \int_{z >z _0} e^{-\phi(z)} |\uu|^2 \, dz \right)\in L^2(0,+\infty)\, . \]
\end{thm}

\begin{rmq}
\begin{enumerate}
\item
The additional weight $e^{-\phi(z)}$ in the last contribution (weighted $L^2$ space) is  specific to the lack of coercivity in the energy estimates. 
%It is responsible for doubling the length of the proof. 
\item The constant $c(\e)$ that we obtain degenerates as $\e\to 1/\sqrt{F'(0)}$, due to the transition from a parabolic to an hyperbolic regime. 
\item We restrict ourselves to compactly supported initial perturbations $u^0$ to justify all integration by parts. Indeed the solution $u(t,z)$ remains compactly supported for all $t>0$ because of the finite speed of propagation (see the kinetic formulation \eqref{eq:twospeeds} and \cite[Chapter 12]{Evans-PDE}). The result would be the same if we were assuming that $u^0$ decays sufficiently fast at infinity.
\end{enumerate}
\end{rmq}

\begin{proof} We proceed in several steps.

\paragraph{1- Derivation of the energy estimates.}  

The equation satisfied by the perturbation $\w$ writes
\begin{multline}
\e^2 \left( \partial_{tt} \w -2 s \partial_{tz } \w + s^2 \partial_{z z } \w  \right) + \left( 1 - \e^2 F'(\nu + \w) \right)\left( \partial_t\w - s\partial_z  \w \right) - \partial_{z z } \w \\
+  \e^2 (F'(\nu + \w) - F'(\nu)) s \partial_{z } \nu  = F(\nu + \w) - F(\nu)\, .
\label{eq:NL stab w1}
\end{multline}
We write the nonlinearities as follows:
\begin{align*} 
& F'(\nu + \w) = F'(\nu) + K_1(z ;\w) \w\, , \\
&  F'(\nu + \w) - F'(\nu)  = F''(\nu) \w + K_2(z ;\w) \w^2\, , \\ 
&  F(\nu + \w) - F(\nu)  = F'(\nu) \w + K_3(z ;\w) \w^2\, .
  \end{align*}
where the functions $K_i$ are uniformly bounded in $L^\infty(\R)$. More precisely we have 
\begin{equation*} 
K_1(z ;\w) = \int_0^1 F''(\nu + t \w)dt\, ,   \; K_2(z ;\w) = \int_0^1 (1-t) F'''(\nu + t \w)dt \, , \; K_3(z ;\w) = \int_0^1 (1-t) F''(\nu + t \w)dt\, .
%%\label{eq:K}
\end{equation*}
Thus we can decompose equation \eqref{eq:NL stab w1} into linear and nonlinear contributions:
\begin{multline}
\e^2 \left( \partial_{tt} \w -2 s \partial_{tz } \w + s^2 \partial_{z z } \w  \right) + \left( 1 - \e^2 F'(\nu) \right) \left( \partial_t\w - s\partial_z  \w \right)   - \partial_{z z } \w 
+ \left(s \e^2 F''(\nu) \partial_{z } \nu   - F'(\nu) \right) \w  \\
=\e^2 K_1(z ;\w) \w \left( \partial_t\w - s\partial_z  \w\right) + \left( K_3(z ;\w) - s \e^2 K_2(z ;\w)\partial_z  \nu \right) \w^2 \, .
\label{eq:NL stab w}
\end{multline}
Testing  equation \eqref{eq:NL stab w} against $ \partial_t\w - s\partial_z  \w$ yields our first energy estimate (hyperbolic energy):
\begin{multline}
\dfrac{d}{dt} \left\{ \dfrac{\e^2}2 \int_\R \left| \partial_t\w - s\partial_z  \w \right|^2\, dz   + \dfrac12 \int_\R \left| \partial_z  \w \right|^2\, dz  + \dfrac12 \int_\R \left(s \e^2 F''(\nu) \partial_{z } \nu   - F'(\nu) \right) |\w|^2\, dz  \right\} \\
 + \int_\R \left( 1 - \e^2 F'(\nu) \right) \left| \partial_t\w - s\partial_z  \w \right|^2 + \dfrac s2 \int_\R \partial_z  \left( s \e^2 F''(\nu) \partial_{z } \nu   - F'(\nu) \right) |\w|^2 \, dz  \\
  = \e^2 \int_\R K_1(z ;\w) \w \left| \partial_t\w - s\partial_z  \w \right|^2\, dz  + \int_\R \left( K_3(z ;\w) - s e^2 K_2(z ;\w)\partial_z  \nu \right) \w^2 \left( \partial_t\w - s\partial_z  \w \right)\, dz \, .
  \label{eq:Ew1(t)}
\end{multline}
We are lacking coercivity with respect to $H^1$ norm in the energy dissipation. Testing  equation \eqref{eq:NL stab w} against $ \w$ yields our second energy estimate (parabolic energy):
\begin{multline}
\dfrac{d}{dt} \left\{ \e^2 \int_\R \w \left( \partial_t\w - s\partial_z  \w\right)\, dz  + \dfrac12 \int_\R \left( 1 - \e^2 F'(\nu) \right) |\w|^2\, dz  \right\}  \\
 - \e^2 \int_\R |\partial_t \w - s\partial_z  \w|^2\, dz  +  \int_\R |\partial_z  \w|^2\, dz  + \int_\R \left( \dfrac{s \e^2}2 F''(\nu) \partial_{z } \nu   - F'(\nu) \right) |\w|^2 \, dz  \\
 = \e^2 \int_\R K_1(z ;\w) \w^2 \left( \partial_t\w - s\partial_z  \w\right) \, dz  + \int_\R \left( K_3(z ;\w) - s e^2 K_2(z ;\w)\partial_z  \nu \right) \w^3 \, dz \, .
  \label{eq:Ew2(t)}
\end{multline}
We introduce the following notations for the two energy contributions and the respective quadratic dissipations \eqref{eq:Ew1(t)},\eqref{eq:Ew2(t)}:
\begin{align*}
E^\w_1(t) & = \dfrac{\e^2}2 \int_\R \left| \partial_t\w - s\partial_z  \w \right|^2\, dz   + \dfrac12 \int_\R \left| \partial_z  \w \right|^2\, dz  + \dfrac12 \int_\R \left(s \e^2 F''(\nu) \partial_{z } \nu   - F'(\nu) \right) |\w|^2\, dz \, ,  \\
E^\w_2(t) &= \e^2 \int_\R \w \left( \partial_t\w - s\partial_z  \w\right)\, dz  + \dfrac12 \int_\R \left( 1 - \e^2 F'(\nu) \right) |\w|^2\, dz \, , \\
Q^\w_1(t) &= \int_\R \left( 1 - \e^2 F'(\nu) \right) \left| \partial_t\w - s\partial_z  \w \right|^2 + \dfrac s2 \int_\R \partial_z  \left( s \e^2 F''(\nu) \partial_{z } \nu   - F'(\nu) \right) |\w|^2 \, dz \, , \\
Q^\w_2(t) &= - \e^2 \int_\R |\partial_t \w - s \partial_z |^2\, dz   +  \int_\R |\partial_z  \w|^2\, dz  + \int_\R \left( \dfrac{s \e^2}2 F''(\nu) \partial_{z } \nu   - F'(\nu) \right) |\w|^2 \, dz \, .
\end{align*}
The delicate issue is to control the zeroth-order terms. In particular we define the weights 
\begin{align*}
A_1(z ) &= s \e^2 F''(\nu) \partial_{z } \nu   - F'(\nu)\, ,\\
A_2(z ) &= \dfrac{s \e^2}2 F''(\nu) \partial_{z } \nu   - F'(\nu)\, .
\end{align*} 
They change sign over $\R$. More precisely we have
\begin{align*}
& \lim_{z \to - \infty} A_1(z ) = \lim_{z \to - \infty} A_2(z ) = - F'(1)>0\, , \\
& \lim_{z \to + \infty} A_1(z ) = \lim_{z \to + \infty} A_2(z ) = - F'(0)<0\, .
\end{align*}

To circumvent this issue we introduce $\uu(t,z ) = e^{\phi(z)} \w(t,z )$ as in \cite{GallayRaugel} and the previous Section \ref{sec:linstab}, where $\phi(z ) $ is a  weight to be determined later \eqref{eq:weight h}.
%We expect $\phi(z )$ to behave like $ -\lambda z  $ as $z  \to +\infty$. 
The new function $\uu(t,z )$ satisfies the following equation:
%\begin{multline}
%\e^2 \left( \partial_{tt} \uu -2 s \partial_{tz } \uu - 2 s \partial_t \uu \partial_z  \phi  \right) + \left( 1 - \e^2 F'(\nu + h \uu)) \right)\left( \partial_t\uu - s\partial_z  \uu - s \uu \partial_z  \phi \right)    \\
%%
%+ (\e^2 s^2 - 1) \left( \partial_{z z } \uu + 2 \partial_z  \uu \dfrac{\partial_{z } h}{h} + \uu \dfrac{\partial_{z z } h}{h}\right) + \e^2 s  \dfrac{F'(\nu + h \uu) - F'(\nu)}h \partial_{z } \nu  = \dfrac{F(\nu + h \uu) - F(\nu)}{h}\, .
%\label{eq:NL stab \uu}
%\end{multline}
\begin{multline}
\e^2  \partial_{tt} \uu -2 \e^2 s \partial_{tz } \uu + \left(  2 \e^2 s \partial_z  \phi +  1 - \e^2 F'(\nu) \right) \partial_t \uu  + \left( - s \left( 1 - \e^2 F'(\nu)  \right) - 2 (\e^2 s^2 -1 ) \partial_z  \phi  \right) \partial_z  \uu      \\ 
+ (\e^2 s^2 - 1)  \partial_{z z } \uu   
+ \left(   s \left( 1 - \e^2 F'(\nu) \right)\partial_z  \phi + (\e^2 s^2 - 1)\left( - \partial_{z z } \phi + |\partial_z  \phi|^2\right) + \e^2 s   F''(\nu)  \partial_{z } \nu - F'(\nu) \right) \uu       \\
= \e^2 K_1(z ;u)u \left( \partial_t\uu - s\partial_z  \uu \right) +  \left(   K_3(z ;u) - \e^2 s K_2(z ;u) \partial_{z } \nu   +  \e^2 s K_1(z ;u)\partial_z  \phi \right) u \uu  \, .
\label{eq:NL stab u}
\end{multline}
We denote the prefactors of $\partial_t \uu$, $\partial_z  \uu$ and $\uu$ as $A_3$, $A_4$ and $A_5$ respectively:
\begin{align}
& A_3(z ) =   2 \e^2 s \partial_z  \phi +  1 - \e^2 F'(\nu)\,,  \nonumber 
%\label{eq:A1}
\\ 
& A_4(z ) = - s \left( 1 - \e^2 F'(\nu)  \right) - 2 (\e^2 s^2 - 1) \partial_z  \phi \,, \label{eq:A2}\\
& A_5(z ) =   s \left( 1 - \e^2 F'(\nu) \right)\partial_z  \phi + (\e^2 s^2 - 1)\left(- \partial_{z z } \phi + |\partial_z  \phi|^2\right) + \e^2 s   F''(\nu)  \partial_{z } \nu - F'(\nu)\, . \nonumber%\label{eq:A3}
\end{align}

%We notice that the prefactor of the linear zeroth-order contributions is nonnegative since $F$ is a concave f\unction ($F(\nu)/\nu \geq F'(\nu)$) and $\nu$ is a nonincreasing profile ($F''(\nu) \partial_{z } \nu \geq 0$). This is the p\uurpose of the transformation $\w = \nu \uu$. Indeed the corresponding prefactor is changing sign over $\R$ in the equation satisfied by $\w$ \eqref{eq:NL stab \w}.
 
Testing \eqref{eq:NL stab u} against $\partial_t \uu$, we obtain our third energy estimate:
\begin{multline*}%\label{ut}
\frac{d}{dt} \left\{\dfrac{\epsilon^2}2 \int_{\R}   | \partial_t \uu |^2  \, dz + \dfrac{1 - \epsilon^2 s^2}2 \int_{\R}  | \partial_z  \uu |^2  \, dz  + \dfrac12\int_{\R}  A_5(z )   | \uu |^2  \, dz  
%+ \int_{\R}  K_4(z ;h \uu) h \uu^3 \, dz  
\right\} \\ 
+ \int_{\R} A_3(z ) | \partial_t \uu |^2 \, dz  + \int_{\R} A_4(z )  \partial_t \uu \partial_z  \uu \, dz  \\ 
=  \e^2\int_{\R} K_1(z ;u) u \left( |\partial_t\uu|^2 - s \partial_t \uu \partial_z  \uu \right)\, dz  + \int_\R \left(   K_3(z ;u) - \e^2 s K_2(z ;u) \partial_{z } \nu   + \e^2 s K_1(z ;u)\partial_z  \phi \right) u w \partial_t \uu \, dz  \, ,
\end{multline*}
%where $K_4$ satisfies: $\partial_\uu \left( K_4(z  ; h \uu) h \uu^3\right) = \left( - K_3(z ;h\uu) + \e^2 s K_2(z ;h\uu) \partial_{z } \nu   +  \e^2 s K_1(z ;\w)\partial_z  \phi \right) h \uu^2$. The f\unction $K_4$ is bo\unded in $L^\infty$ provided $\partial_z  \phi\in L^\infty$ since $\partial_\uu(K_4(z ;h\uu) \uu^3 ) = \mathcal O( |\uu|^2)$. This is indeed the case \eqref{eq:weight h}. Hence the nonlinear contributions in the energy behaves essentially as a c\uubic term with respect to $\uu$. 
%The energy estimate \eqref{\uut} rewrites as follows:
%\[ \dfrac d{dt} E_1(t) + Q_1(t) = \mbox{c\uubic terms}\, . \]
%We notice that the prefactor of $|\w|_{L^2}^2$ in the energy $F_1$ is problematic since $(2\nu(z ) - 1 - 2 \e^2 s \partial_{z } \nu(z ))$ changes sign as $z $ goes over $\R$. 
Testing \eqref{eq:NL stab u} against $\uu$ we obtain our last energy estimate:
\begin{multline*}%\label{u}
\frac{d}{dt} \left\{\epsilon^2 \int_{\R}  \uu\partial_t \uu \, dz + \dfrac12 \int_{\R} A_3(z )  | \uu |^2  \, dz  \right\} \\
- \epsilon^2  \int_{\R} | \partial_t \uu |^2 \, dz  + (1 - \epsilon^2 s^2) \int_{\R} | \partial_z  \uu |^2 \, dz  +  2s\epsilon^2 \int_{\R} \partial_t \uu \partial_z  \uu \, dz  + \int_{\R} \left( A_5(z ) - \dfrac{\partial_z  A_4(z )}{2} \right) | \uu |^2 \, dz  \\
=  \int_\R \left(   K_3(z ;u) - \e^2 s K_2(z ;u) \partial_{z } \nu   +  \e^2 s K_1(z ;u)\partial_z  \phi  \right) u w^2 \, dz  + \int_\R \e^2 K_1(z ;u) u w \left( \partial_t\uu - s\partial_z  \uu \right) \, dz  \, ,
\end{multline*}
%where $L$ satisfies $\partial_\uu L(z ,\uu) = - \e^2 J(z ,\uu) h \uu^2$. It behaves essentially as a c\uubic term with respect to $\uu$. 
We introduce again useful notations for the two energy contributions and the associated quadratic dissipations:
\begin{align}
& E^\uu_1(t) = \dfrac{\epsilon^2}2 \int_{\R}   | \partial_t \uu |^2  \, dz + \dfrac{1 - \epsilon^2 s^2}2 \int_{\R}  | \partial_z  \uu |^2  \, dz  + \dfrac12\int_{\R}  A_5(z )   | \uu |^2  \, dz  
%+ \int_{\R}  K_4(z ;h \uu) h \uu^3 \, dz  
\,,\label{eq:Eu1} \\
& E^\uu_2(t) = \epsilon^2 \int_{\R}  \uu\partial_t \uu \, dz + \dfrac12 \int_{\R} A_3(z )  | \uu |^2 \, dz   \, , \label{eq:Eu2} \\
& Q^\uu_1(t) = \int_{\R} A_3(z ) | \partial_t \uu |^2 \, dz  + \int_{\R} A_4(z )  \partial_t \uu \partial_z  \uu \, dz \,. \nonumber %\label{eq:Qu1} 
\\
& 
Q^\uu_2(t) = - \epsilon^2  \int_{\R} | \partial_t \uu |^2 \, dz  + (1 - \epsilon^2 s^2) \int_{\R} | \partial_z  \uu |^2 \, dz  +  2 \epsilon^2 s \int_{\R} \partial_t \uu \partial_z  \uu \, dz  + \int_{\R} \left( A_5(z ) - \dfrac{\partial_z  A_4(z )}{2} \right) | \uu |^2 \, dz \, . \label{eq:Qu2}
\end{align}
%The energy estimate \eqref{\w} rewrites in short:
%\[ \dfrac d{dt} E_2(t) + Q_2(t) = \mbox{c\uubic terms}\, . \]

To determine $\phi(z )$ we examinate \eqref{eq:Eu1}--\eqref{eq:Eu2}. We first require the natural condition $\partial_z  \phi(z ) \geq 0$.
%, which is motivated by the analogy between $h(z )$ and $\nu(z )$ (altho\uugh the choice $h(z ) = \nu(z )$ does not fit). 
This clearly ensures $A_3(z )\geq 1 - \e^2 F'(0)$. %If we take for granted that $\log h(z )$ is a concave f\unction, then we also have directly $\partial_z  A_2(z ) = \e^2 s F''(\nu) \partial_z  \nu + 2(\e^2 s^2 - 1) \partial_{z z }\log h(z ) \geq 0 $.
%
%The t\wo conditions are now $A_3(z ) \geq 0$ \eqref{eq:E1} and 
We examinate the condition $A_5(z ) -  \frac12\partial_z  A_4(z )\geq 0$ \eqref{eq:Qu2} in order to fully determine the weight $\phi(z)$:
\begin{align*}
A_5(z ) - \dfrac{\partial_z  A_4(z )}{2} &=   s \left( 1 - \e^2 F'(\nu) \right)\partial_z  \phi + (\e^2 s^2 - 1)\left( - \partial_{z z } \phi + |\partial_z  \phi|^2\right) + \e^2 s   F''(\nu)  \partial_{z } \nu - F'(\nu) \\
& \quad - \dfrac12 \e^2 s F''(\nu) \partial_z  \nu + (\e^2 s^2 - 1)     \partial_{z z } \phi   
\\
& = (\e^2 s^2 - 1) \left|\partial_z  \phi\right|^2 + s \left( 1 - \e^2 F'(\nu) \right)\partial_z  \phi   + \dfrac12 \e^2 s F''(\nu) \partial_z  \nu - F'(\nu)\, .
\end{align*}
This is a second-order equation in the variable $\partial_z  \phi $. %Again by the analogy between $h(z )$ and $\nu(z )$ we impose $\lim_{z \to -\infty}\partial_z  \log h(z ) = 0$ and $\lim_{z \to +\infty}\partial_z  \log h(z ) = -\lambda$. For this we solve explicitely the equation 
%\[ A_3(z ) - \dfrac{\partial_z  A_2(z )}{2} - \dfrac12 \e^2 s F''(\nu) \partial_z  \nu = G(\nu) \, ,  \]
%where $G\geq 0$ satisfies $G(0) = 0$ and $G(1) = -F'(1)$. In the classical case $F(r) = r(1 - r)$ we shall opt for $G(\nu) = \nu$. 
%The discriminant writes
Maximization of this quantity is achieved when
\begin{equation} \partial_z  \phi  = \dfrac{s\left( 1 - \e^2 F'(\nu) \right)}{2(1 - \e^2 s^2 )} \, \geq \, 0\, . \label{eq:weight h}
\end{equation} 
We notice that this is equivalent to setting $A_4(z ) = 0$ \eqref{eq:A2}. Then we obtain
\begin{align*}
A_5(z ) & = \dfrac{  s^2 \left(1 - \e^2F'(\nu)\right)^2}{4(1 - \e^2 s^2)}   
+ \dfrac12 \e^2 s F''(\nu) \partial_z  \nu - F'(\nu)\\
& = \dfrac{ 1}{4(1 - \e^2 s^2)} \left( s^2 \left(1 + \e^2F'(\nu)\right)^2  - 4 F'(\nu)\right)+ \dfrac12 \e^2 s F''(\nu) \partial_z  \nu  \\
& = \dfrac{1}{4\left(1 - \e^2 F'(0)\right)^2} 
\left(  4F'(0)  \left(1 + \e^2F'(\nu)\right)^2 - 4 F'(\nu) \left(1 + \e^2 F'(0)\right)^2 \right) + \dfrac12 \e^2 s F''(\nu) \partial_z  \nu \\
& =   \dfrac{1}{\left(1 - \e^2 F'(0)\right)^2} \left( F'(0) - F'(\nu)\right) \left( 1 - \e^4 F'(0) F'(\nu)\right)  + \dfrac12 \e^2 s F''(\nu) \partial_z  \nu\, .
\end{align*}
We check that $A_5(z )\geq 0$ since $\forall z \; F'(\nu(z )) \leq F'(0) $, $\e^2 F'(0)<1$, $\forall z \; F''(\nu(z ))\leq 0$ and $\forall z \; \partial_z~\nu(z)\leq0$. 

We recall that the exponential decay of $\nu$ at $+\infty $ is given by the eigenvalue $\lambda>0$, where
\[ \lambda = \dfrac{s(1 - \e^2 F'(0))}{2(1 - \e^2 s^2 )}\, . \]
Therefore we can rewrite
\begin{equation} 
\partial_z  \phi =   \lambda \frac{1 - \e^2 F'(\nu) } { 1 - \e^2 F'(0) }\, . 
\label{eq:weight phi}
\end{equation}
%Hence the weight $h = \exp(\phi)$ has the same exponential decay than $\nu$ at $+\infty$.

\begin{rmq}
As far as we are concerned with linear stability, the energies $E^\uu_1$ and $E^\uu _2$ contain enough information. However proving nonlinear stability requires an additional control of $\w$ in $L^\infty$ which can be obtained using $E^\w_1$ and $E^\w _2$ \cite{GallayRaugel}.
\end{rmq}

%%%%%%%%%

\paragraph{2- Combination of the energy estimates.} 
We first examinate the energies $E_1^\w$ and $E_2^\w$. We clearly have
\begin{align*}
E^\w_2(t) & \geq - \dfrac{\e^4}{1 - \e^2 F'(0)}  \left\| \partial_t\w - s\partial_z  \w \right\|^2_2 -    \dfrac{1 - \e^2 F'(0)}4   \|\w\|_2^2  + \dfrac{1 - \e^2 F'(0)}2 \|\w\|_2^2 \\
& \geq - \dfrac{\e^4}{1 - \e^2 F'(0)}  \left\| \partial_t\w - s\partial_z  \w \right\|^2_2 +     \dfrac{1 - \e^2 F'(0)}4   \|\w\|_2^2
\end{align*}
We set 
\begin{equation*}
\delta = \dfrac{1 - \e^2 F'(0)}{2 \e^2}\, .
\end{equation*}
We have on the one hand
\begin{align*}
E^\w_1(t) + \delta E^\w_2(t) &\geq 
%\left( \dfrac{\e^2}2 - \delta \dfrac{\e^4}{1 - \e^2 F'(0)} \right) \left\| \partial_t\w - s\partial_z  \w \right\|^2_2 + \dfrac12 \int_\R \left| \partial_z  \w \right|^2\, dz  + \dfrac12 \int_\R \left( A_1(z ) + \delta (1 - \e^2 F'(\nu))\right) |\w|^2\, dz  \\
\dfrac12  \left\| \partial_z  \w \right\|_2^2\, dz  +   \int_\R A_6(z ) |\w|^2\, dz \, ,
\end{align*}
where $A_6(z )$ is defined as
\[ A_6(z ) =  \dfrac12 A_1(z ) + \delta \dfrac{ 1 - \e^2 F'(0)}4\, . \]
We have on the other hand,
\begin{align*}
Q^\w_1(t) + \delta Q^\w_2(t) &\geq 
\dfrac{1 - \e^2 F'(0)}2 \left\| \partial_t\w - s\partial_z  \w \right\|^2_2 + \delta \|\partial_z  \w\|_2^2 + \int_\R A_7(z )|\w|^2 \, dz \, , 
\end{align*}
where $A_7(z )$ is defined as
\[ A_7(z ) = \dfrac s2  \partial_z  \left( s \e^2 F''(\nu) \partial_{z } \nu   - F'(\nu) \right) + \delta A_2(z )\, . \]

We have both $\lim_{z \to -\infty} A_6(z ) >0$ and $\lim_{z \to -\infty} A_7(z ) >0$. Accordingly there exists $\alpha>0$ and $z _0\in \R$ such that 
\[ \forall z < z _0\quad \min(A_6(z ), A_7(z )) > \alpha\, . \]

In order to control the zeroth-order terms over $(z _0,+\infty)$ we shall use the last two energy estimates. 
First we observe that $\forall z > z _0 \; |\w(z )| = |e^{-\phi(z)} \uu(z )| \leq e^{-\phi(z_0)} |\uu(z )|$ since $\phi$ is increasing. We set $\phi(z _0) = 0$ without loss of generality. This determines completely $ \phi $ together with the condition \eqref{eq:weight h}.
We have
\begin{align*}
 E^\w_1(t) + \delta E^\w_2(t) &\geq \dfrac12  \left\| \partial_z  \w \right\|_2^2\, dz  +   \alpha \int_{z < z _0}  |\w|^2\, dz  - \left\|A_6 e^{-2\phi} \ind_{z >z _0} \right\|_\infty \int_{z >z _0} |\uu|^2\, dz  \, , \\
  Q^\w_1(t) + \delta Q^\w_2(t) &\geq \dfrac{1 - \e^2 F'(0)}2 \left\| \partial_t\w - s\partial_z  \w \right\|^2_2 + \delta \|\partial_z  \w\|_2^2 +  \alpha \int_{z < z _0}  |\w|^2\, dz  - \left\|\dfrac{A_7 e^{-2\phi}}{A_5}\ind_{z >z _0}\right\|_\infty \int_{z >z _0} A_5(z )|\uu|^2 \, dz \, .
\end{align*}

\begin{lem}
We have $\dfrac{A_7 e^{-2\phi}}{A_5}  \in L^\infty(z _0,+\infty)$ and $\dfrac{ e^{-\phi} }{A_5}  \in L^\infty(z _0,+\infty)$.
\end{lem}
\begin{proof}
The first claim is clearly a consequence of the second claim since $A_7 e^{-\phi} \in L^\infty(z _0,+\infty)$. First we have 
\[ F'(0) - F'(\nu) \geq \left( \inf_{[0,1]}\left( - F'' \right) \right) \nu\, = \alpha \nu\;, \]
where $\alpha > 0$ is the coercivity constant of $-F$ \eqref{eq:hyp F}. As a consequence, $A_5 \geq \left( \frac{\left( 1 + \e^2 F'(0) \right) \alpha }{1 - \e^2 F'(0)} \right) \nu $.
Second we recall $\partial_z  \nu + \lambda \nu \geq 0$ (Lemma \ref{lem:varphi<lambda}), so that $\forall z > z _0,\; \nu(z ) \geq \nu(z_0) e^{-\lambda ( z - z_0) }$. Finally we have $\forall z,  \; \partial_z  \phi \geq  \lambda$ \eqref{eq:weight phi}, thus
\[ \forall z  > z _0, \quad   e^{-\phi(z)} \leq e^{- \phi(z_0)} e^{-\lambda(z  - z _0)}\leq \frac{e^{- \phi(z_0)}}{\nu(z_0)} \nu(z) \leq \frac{e^{- \phi(z_0)}}{\nu(z_0)}  \left( \frac{1 - \e^2 F'(0)}{\left( 1 + \e^2 F'(0) \right) \alpha } \right)  A_5(z) \, . \]
\end{proof}

We now focus on the second series of energy estimates. We clearly have
\begin{align*}  
E_2^\uu(t) & \geq - \dfrac{\e^4}{1 - \e^2 F'(0)}  \left\| \partial_t \uu \right\|^2_2 -     \dfrac{1 - \e^2 F'(0)}4   \|\uu\|_2^2  + \dfrac{1 - \e^2 F'(0)}2 \|\uu\|_2^2 \\
& \geq - \dfrac{\e^4}{1 - \e^2 F'(0)}  \left\| \partial_t\uu \right\|^2_2 +     \dfrac{1 - \e^2 F'(0)}4   \|\uu\|_2^2\, ,
\end{align*}
and 
\begin{align*}
Q_2^\uu(t) & \geq - \e^2 \|\partial_t \uu\|_2^2  + (1 - \e^2 s^2)\|\partial_z  \uu\|_2^2 - \dfrac{2\e^4 s^2}{1 - \e^2 s^2} \|\partial_t \uu\|_2^2  - \dfrac{1 - \e^2 s^2}2 \|\partial_z  \uu\|_2^2 + \int_{\R}   A_5(z )   | \uu |^2 \, dz  \\
& \geq - \e^2 \dfrac{1 + \e^2 s^2}{1 - \e^2 s^2}\|\partial_t \uu\|_2^2 + \dfrac{1 - \e^2 s^2}2 \|\partial_z  \uu\|_2^2 + \left\| \dfrac {e^{-\phi}}{A_5} \ind_{z >z _0} \right\|_\infty^{-1} \int_{z >z _0}   e^{-\phi}   | \uu |^2 \, dz \, .
\end{align*}
We set
\[ \delta' =  \dfrac{1 - \e^2 F'(0)}{2\e^2} \cdot \dfrac{1 - \e^2s^2}{1 + \e^2s^2} < \delta \, .\]
We have on the one hand
\begin{align*}
E^\uu_1(t) + \delta' E^\uu_2(t) &\geq 
%\left( \dfrac{\e^2}2 - \delta \dfrac{\e^4}{1 - \e^2 F'(0)} \right) \left\| \partial_t\w - s\partial_z  \w \right\|^2_2 + \dfrac12 \int_\R \left| \partial_z  \w \right|^2\, dz  + \dfrac12 \int_\R \left( A_1(z ) + \delta (1 - \e^2 F'(\nu))\right) |\w|^2\, dz  \\
\dfrac{1 - \e^2 s^2}2  \left\| \partial_z  \uu \right\|_2^2\, dz  +   \int_\R A_8(z ) |\uu|^2\, dz \, ,
\end{align*}
where $A_8(z )$ is defined as
\[ A_8(z ) =  \dfrac12 A_5(z ) + \delta' \dfrac{ 1 - \e^2 F'(0)}4 \geq \delta' \dfrac{ 1 - \e^2 F'(0)}4 \, . \]
We have on the other hand,
\begin{align*}
Q^\uu_1(t) + \delta' Q^\uu_2(t) &\geq 
\dfrac{1 - \e^2 F'(0)}2 \left\| \partial_t\uu \right\|^2_2 + \delta' \dfrac{1 - \e^2 s^2}2 \|\partial_z  \uu\|_2^2  + \delta' \left\| \dfrac {e^{-\phi}}{A_5} \ind_{z >z _0} \right\|_\infty^{-1} \int_{z >z _0} h |\uu|^2 \, dz \, . 
\end{align*}

Combining all these estimates, we define $E(t) = E_1^\uu(t) + \delta' E_2^\uu(t) + \delta''\left( E_1^\w(t) + \delta E_2^\w(t) \right)$ and $Q(t) = Q_1^\uu(t) + \delta' Q_2^\uu(t) + \delta''\left( Q_1^\w(t) + \delta Q_2^\w(t) \right)$, where $\delta''>0$ is defined such as the following condition holds true
\[ \delta'' < \delta' \min\left(   \dfrac{1 - \e^2 F'(0)}4 \left\|A_6 e^{-2\phi} \ind_{z >z _0} \right\|_\infty^{-1},   \left\|\dfrac{A_7 e^{-2\phi}}{A_5}\ind_{z >z _0}\right\|_\infty^{-1} \right) \, .  \]
 finally obtain our main estimate,
\begin{multline}
\dfrac d{dt} E(t) + Q(t)   \leq \mathcal O\left( \int_\R |\w||\partial_t \w - s\partial_z  \w|^2\, dz   + \int_\R |\w|^3\, dz  \right) \\ + \mathcal O\left( \int_\R e^{-\phi}|\uu||\partial_t \uu|^2\, dz  + \int_\R e^{-\phi}|\uu||\partial_z  \uu|^2\, dz  + \int_\R e^{-\phi} |\uu|^3\, dz   \right)\, , 
\label{eq:main}
\end{multline}
where 
\begin{align*}
E(t) & \geq \mathcal O\left( \|\partial_z  \w\|_2^2 + \int_{z  < z _0} |\w|^2\, dz  +   \|\partial_z  \uu\|_2^2 + \|\uu\|_2^2\right)\, ,  \\
Q(t) & \geq \mathcal O\left( \|\partial_t - s\partial_z  \w\|_2^2 + \|\partial_z  \w\|_2^2  + \int_{z  < z _0} |\w|^2\, dz  + \|\partial_t \uu\|_2^2 + \|\partial_z  \uu\|_2^2 + \int_{z >z _0} e^{-\phi} |\uu|^2\, dz  \right)\, .
\end{align*}

\paragraph{3- Control of the nonlinear contributions.}
Our goal is to control the size of the perturbation $\w$ in $L^\infty$. For this purpose we use the embeddings of $H^1(\R)$ into $L^\infty(\R)$ and $L^4(\R)$: 
\begin{align*}
\|\w \|_{\infty}  & \leq C \|\w\|_2^{1/2} \|\partial_z  \w\|_2^{1/2} \, ,
%\label{eq:H1 embedding}
\\
\|\w \|_{4} & \leq C \|\w\|_2^{3/4} \|\partial_z  \w\|_2^{1/4} \, .
%\label{eq:H1 embedding bis}
\end{align*}
We examinate successively the nonlinear contributions. We recall $\w = e^{-\phi} \uu$. First we have
\begin{align*}  
\int_\R |\w||\partial_t \w - s\partial_z  \w|^2\, dz  & \leq \|\w\|_\infty \|\partial_t - s\partial_z  \w\|_2^2 \\
& \leq \mathcal O\left( E(t)^{1/2} Q(t) \right)\,,  
\end{align*}
and similar estimates can be derived for all the contributions in the r.h.s. of \eqref{eq:main} except for the last one. 
Second we have
\begin{align*}
\int_\R e^{-\phi} |\uu|^3\, dz  & = \int_{z <z _0} e^{-\phi} |\uu|^3\, dz  + \int_{z >z _0} e^{-\phi} |\uu|^3\, dz 
\\
& \leq \| \w \|_{L^{\infty}(-\infty,z _0)}\| \uu \|^2_{L^{2}(-\infty,z _0)} + \| e^{-\phi/2}  \uu \|_{L^4(z _0,+\infty)}^{2}\|\uu\|_{L^2(z _0,+\infty)} \\
& \leq C \| \w \|_{L^{2}(-\infty,z _0)}^{1/2} \| \partial_z  \w \|_{L^{2}(-\infty,z _0)}^{1/2} \| \uu \|^2_{L^{2}(-\infty,z _0)} 
\\& \quad + C \| e^{-\phi/2} \uu \|_{L^2(z _0,+\infty)}^{3/2}   
\left\| \partial_z \left( e^{-\phi/2} \uu\right) \right\|^{1/2}_{L^2(z _0,+\infty)}
\|\uu\|_{L^2(z _0,+\infty)}\, .
\end{align*}
We have 
\begin{align*}
\| \uu \|^2_{L^{2}(-\infty,z _0)} &   \leq \|   \w \|^2_{L^{2}(-\infty,z _0)} \,, \\
\| e^{-\phi/2} \uu \|^2_{L^2(z _0,+\infty)} & = \int_{z >z _0} e^{-\phi} |\uu|^2\, dz   \, , \\
\left\| \partial_z \left( e^{-\phi/2} \uu\right) \right\|^2_{L^2(z _0,+\infty)} &\leq  2
\int_{z >z _0} \left(  e^{-\phi}  |\partial_z  \uu|^2 + \dfrac14 |\partial_z  \phi|^2 e^{-\phi} |\uu|^2 \right)\, dz  \leq 2 \int_{z >z _0}  |\partial_z  \uu|^2\, dz  + C \int_{z >z _0} e^{-\phi} |\uu|^2\, dz  \,.
\end{align*}
Consequently we obtain
\begin{equation*}
\int_\R e^{-\phi} |\uu|^3\, dz  \leq \mathcal O\left( E(t)^{1/2} Q(t) \right)\, . 
\end{equation*}

Finally we get
\begin{equation*}
\dfrac d{dt} E(t) + Q(t)   \leq \mathcal O\left(E(t)^{1/2} Q(t) \right)\, . 
\end{equation*}
This estimate ensures that the energy is nonincreasing provided that it is initially small enough. Indeed there exists a constant $C$ such that $\frac d{dt} E(t) + Q(t) \leq C E(t)^{1/2} Q(t)$. We set $c = C^{-2}/2$. If initially $E^0\leq c$ then the previous differential inequality guarantees that $E(t)$ is decaying and remains below the level $c$. Therefore $E(t)$ is positive decaying, and the dissipation $Q(t)$ is integrable. This concludes the proof of Theorem \ref{thm:NLstab}.
\end{proof}

\end{document}